\documentclass[12pt]{article}
\usepackage{color}
\usepackage{dsfont}
\usepackage{enumerate}
\usepackage{graphicx,float}
\usepackage{caption}
\usepackage{subcaption}
\usepackage{setspace}
\usepackage{hyperref}
\usepackage{color}
\usepackage{diagbox}
\usepackage{comment}
\usepackage{amsmath, amsfonts, amssymb, amsthm, amscd,graphicx}
\usepackage[ruled,vlined]{algorithm2e}
\usepackage{authblk}

\SetKwInput{KwInitialize}{Initialization}
\SetKwInput{KwOutput}{Output}

\newtheorem{Theorem}{Theorem}[part]
\newtheorem{Definition}{Definition}[part]
\newtheorem{Proposition}{Proposition}[part]

\newtheorem{Lemma}{Lemma}[part]

\newtheorem{Remark}{Remark}[part]

\def \R{\mathbb{R}}

\def \E{\mathbb{E}}

\def \P{\mathbb{P}}

\def \Ac{{\cal A}}

\def \Bc{{\cal B}}

\def \Pc{{\cal P}}

 \def \Nc{{\cal N}}

\def \Wc{{\cal W}}

\def \1{{\mathds 1}}

\def \eps{\varepsilon}

\def\beqs{\begin{eqnarray*}}
\def\enqs{\end{eqnarray*}}
\def\beq{\begin{eqnarray}}
\def\enq{\end{eqnarray}}

 \title{Discrete-Time Mean-Field Stochastic Control with Partial Observations}

\author[,1,2]{Jeremy Chichportich\thanks{Email: \texttt{jeremy@bramham-gardens.com}}}
\author[,1]{Idris Kharroubi\thanks{Corresponding author. Email: \texttt{idris.kharroubi@sorbonne-universite.fr}. The  author gratefully acknowledge financial support from the Agence Nationale de la Recherche (ReLISCoP grant ANR-21-CE40-0001). }}
 \affil[1]{LPSM, UMR CNRS 8001, Sorbonne University and Universit\'e de Paris}
\affil[2]{Bramham gardens}
             \date{\today}

\begin{document}

 \maketitle

 \begin{abstract} We study the optimal control of discrete time mean filed dynamical systems under partial observations. We express the global law of the filtered process as a controlled system with its own dynamics. Following a dynamic programming approach, we prove a verification result providing a solution to the optimal control of the filtered system.  As an application, we study a general linear quadratic example for which an explicit solution is given. We also describe an algorithm for the numerical approximation of the optimal value and provide numerical experiments on a financial example.
\end{abstract}


\section{Introduction}
We consider an optimal control problem for a system with mean field discrete time dynamics under partial observations. The question of how to optimally control a system with mean field dependence in the dynamics is related to the modelization of the optimal behavior for populations with large number of interacting individuals.  

The case where each individual chooses an own action in order to optimise a non-cooperative reward, has led to the theory of mean-field games (MFGs), introduced in \cite{HCM06} and \cite{LL07}. The Nash equilibrium is then described by two equations, the first one corresponding to the optimal behavior of a representative agent and the second one corresponding to the evolution of the whole population under the optimal choice of each individual. 

In the case where all the individuals follow the same goal, the cooperative behavior leads to the optimal control of mean field dynamical systems. A growing literature has emerged on the subject for the continuous time case with two main approaches. The first one follows a maximum principle method to provide necessary and sufficient conditions for optimality, see e.g. \cite{AD10, MOX10, BDL11, CD15,JLY15}. The second approach consists  in proving a dynamic programming principle to solve the problem, see e.g. \cite{AX01, BFY13, PW18}.  

Such models with law dependency of the dynamics appear in mathematical finance. The calibration of local volatility models to market smiles leads indeed to mean field stochastic differential equations. We refer to \cite[Chapter 11]{guyon2013nonlinear} for more details.

\vspace{2mm}

Concerning the discrete-time case, \cite{ELY13} studied the case of a linear-quadric problem and turn it into a quadratic optimization problem in an Hilbert space, allowing to get necessary and sufficient conditions for solving the problem. A dynamic programming approach is use in \cite{PW16} to the case where controls are restricted to feedback ones.  By considering the law of the controlled process as a state variable, it allows to get a verification theorem and to solve explicitly the linear quadratic case.

\vspace{2mm}

In this paper, we investigate the optimal control problem of mean-field discrete time systems under partial information. 
This question has already been studied for systems without mean-field interaction. We refer to the book \cite{EAM95} for a detailed presentation of the estimation and the control of systems under partial information. The common approach to deal with optimal control of partially observable systems consists in computing the conditional law of the unobserved component given the observation and to derive its dynamic to retrieve a completely observable controlled system called filtered controlled system.   

\vspace{2mm}

In the case where the unobservable component admits a mean-field dependence, this approach leads to an intractable problem since the filtered system involves two different marginal laws, the original law of the unobservable component and the filtered law given the observations. 



%

\vspace{2mm}

To overcome this issue, we consider the law of the global system composed by both observed and unobserved components as a state variable. By keeping in mind the past, it also allows to consider general controls that might not be in the feedback form.
We then rewrite the initial problem as a new control problem with respect to the global law of the system starting from the initial time and with controls as functions depending only on the observed variables. We derive a dynamic programming principle and a verification theorem for this new formulation.  A second verification theorem is also provided for feedback controls. As an application of the verification theorem, we provide the explicit solution of a linear-quadratic optimal control problem with partial observations. 
As it is not possible to get explicit solutions  in many cases, we describe an algorithm for the numerical approximation of the optimal value of our problem. We then provide numerical experiments for this algorithm. 
We first test our algorithm an a benchmark given by the explicit solution of the linear-quadratic model. We then apply our algorithm to a model inspired by the  optimal investment problems for private equities in finance and compare the results with standard strategies.

\vspace{2mm}

The remainder of the paper is organized as follows. In Section \ref{Sec2}, we present the control problem and the computation of the filter. In Section  \ref{Sec3}, we extend the original filtered problem to get a tractable problem. We then provide a dynamic programing principle for the extended problem and a verification theorem. We also present a feedback version of this verification theorem. Finally, we present in Section \ref{Sec4} some applications. We first study a linear quadratic model for which we give an explicit solution via the verification Theorem.  An algorithm is then provided to approximate the optimal valuein the general case. We test this algorithm on the linear-quadratic problem and run an experiment on a financial model inspired by private equity investments.  

\section{The control problem under partial observation}\label{Sec2}
\subsection{The model}
We fix a probability space $(\Omega,\Ac,\P)$ and three Banach spaces $E$, $F$ and $C$. We denote by $|.|$ the norm on those spaces and by $\Bc(E)$, $\Bc(F)$ and $\Bc(C)$ their respective Borel $\sigma$-algebrae. We also denote by $\Pc_2(E)$ the set of probability measure $\mu$ on $(E,\Bc(E))$ such that $\int_E|x|^2d\mu(x)<+\infty$. We similarly define $\Pc_2(E\times E)$ and $\Pc_2(E\times F)$. We endow the set $\Pc_2(E)$ with the $2$-Wasserstein distance $\Wc_2$ defined by
\beqs
\Wc_2(\mu,\mu') & = &\Big( \inf\big\{\int_{E^2}|x-y|^2\pi(dx,dy),\; \pi\in\Pc_2(E^2)~:~\pi(.\times E)=\mu \;,~\pi(E\times .)=\mu'\big\}\Big)^{\frac{1}{2}}
\enqs
and denote by $\Bc(\Pc_2(E))$ its related Borel $\sigma$-algebra. We also similarly define  $\Wc_2$ on $E\times F$ and $\Bc(\Pc_2(E\times F))$. We fix a terminal time $T$ and we define the partially observed control system.

\paragraph{Controls.} A control is a sequence $\alpha:=(\alpha_n)_{0\leq n\leq T-1}$ of random variables defined on  $(\Omega,\Ac,\P)$ and valued in $C$ such that
\beq\label{cond int strat}
\E\big[|\alpha_n|^2\big] & < & +\infty
\enq
for all $n=0,\ldots,T-1$.   We denote by $\mathfrak{C}$ the set of such controls.  

\paragraph{Hidden system.} For a given control $\alpha\in\mathfrak{C}$, we consider a controlled process $(X_k)_{0\leq k\leq T}$ defined by its initial condition  $X_0^\alpha=\xi$ where $\xi\in L^2(\Omega,\Ac,\P;E)$ 
  and the dynamics
\beqs
X_{n+1}^\alpha & = & G_{n+1}(X_{n}^\alpha,\P_{X_{n}^\alpha}, \alpha_{n}, \eps_{n+1})\;,~0\leq n \leq T-1
\enqs
for some measurable functions $G_1,\ldots,G_T$ from $E\times \Pc_2(E)\times C\times \R^d$ to $E$, where $(\eps_n)_{1\leq n\leq T}$ is a sequence of square integrable i.i.d. random variables valued in $\R^d$ and independent of  $\xi$. We make the following assumption on the functions $G_1,\ldots,G_{T}$.

\vspace{2mm}

\noindent \textbf{(H1)} There exist a constant $C$ such that
\beqs
G_{n}(x,\mu, a, e) & \leq & C\big(1+|x|^2+\Wc_2(\mu,\delta_0)^2+|a|^2+|e|^2\big)\;, 
\enqs
for all$(x,a,\mu,e)\in E\times C\times \Pc_2(E)\times \R^d$ and  $n=1,\ldots,T$. 

\vspace{2mm}

\paragraph{Observed System.} For a given control $\alpha\in \mathfrak{C}$, the observation is given by a process $(Y^\alpha_n)_{0\leq n\leq T}$ defined  by its initial condition $Y_0^\alpha = \zeta$ where $\zeta\in L^2(\Omega,\Ac,\P;F)$ is independent of 
 $(\eps_n)_{{1\leq n\leq T}}$  and the dynamics
\beqs
Y^\alpha_{n+1} & = & H_{n+1}(X^\alpha_{n+1},Y^\alpha_{n}, \alpha_n, \eta_{n+1})\;,~0\leq n\leq T-1
\enqs
for some measurable functions $H_1,\ldots,H_T$ from $  E\times F\times C\times \R^{d'}$ into $F$, where $(\eta_n)_{1\leq n\leq T}$ is a sequence of i.i.d. random variables valued in $\R^{d'}$, independent of  $\xi$, $\zeta$ and $(\eps_n)_{1\leq n\leq T}$. 
We make the following assumption on the functions $H_1,\ldots,H_{T}$.

\vspace{2mm}

\noindent \textbf{(H2)} There exist a constant $C$ such that
\beqs
H_{n}(x,y, e) & \leq & C\big(1+|x|^2+|y|^2+|e|^2\big)\;, 
\enqs
for all $(x,y,e)\in E\times F\times  \R^{d'}$ and  $n=1,\ldots,T$. 

\vspace{2mm}

Under Assumptions \textbf{(H1)} and \textbf{(H2)}, we get that the controlled process $(X^\alpha_n,Y^\alpha_n)_{0\leq n\leq T}$ is square integrable
\beqs
\E\big[|X^\alpha_n|^2]+\E\big[|Y^\alpha_n|^2] & < & +\infty
\enqs
for any control $\alpha\in \tilde{\mathfrak{C}}$ and any $n=0,\ldots,T$.

 \paragraph{Optimization problem.} We fix $T$ measurable functions $c_0,\ldots,c_{T-1}$ from $E\times \Pc_2(E)\times C$  to $\R$ and a function $\gamma$ from $E\times \Pc_2(E)$ to $\R$.  on which we make the following assumption.

\vspace{2mm}

\noindent \textbf{(H3)} There exist a constant $C$ such that
\beqs
c_{n}(x,\mu, a) & \leq & C\big(1+|x|^2+\Wc_2^2(\mu,\delta_0)+|a|^2\big)\;,\quad n=0,\ldots,T-1\;,\\
\gamma(x,\mu) & \leq & C\big(1+|x|^2+\Wc_2^2(\mu,\delta_0)  \big)
\enqs
for all $(x,\mu,a)\in E\times \Pc_2(E)\times  C$. 
 
\vspace{2mm}

 \noindent We next introduce the cost function $J$ by 
\beqs
J(\alpha) & := & \E\Big[ \sum_{k=0}^{T-1}c_k\big(X_k^\alpha,\P_{X_k^\alpha},\alpha_k\big)+\gamma\big(X^\alpha_T,\P_{X^\alpha_T}\big) \Big]
\enqs
for  a strategy $\alpha\in\mathfrak{C}$. We notice that under \textbf{(H1)}, \textbf{(H2)} and \textbf{(H3)}  $J(\alpha)$ is well defined.
We now define the subset $\bar{\mathfrak{C}}$ of $\mathfrak{C}$ by
\beqs
\bar{\mathfrak{C}} & = & \Big\{\alpha=(\alpha_k)_{0\leq k\leq T-1}\in {\mathfrak{C}} \text{ adapted to the filtration generated by }  Y^\alpha
\Big\}
\enqs
The problem is to compute
\beq\label{defV0}
V_0 & = & \inf_{\alpha\in \bar{\mathfrak{C}}}J(\alpha)\;.
\enq

 In the sequel we use the following notation : for $N\geq 0$ and $(x_0,\ldots,x_N)$ a given vector, we write $x_{n:m}$ for $(x_n,\ldots,x_m)$ where $0\leq n\leq m\leq N$.

We notice that for $\alpha \in  \bar{\mathfrak{C}}$, there exists measurable functions $a_k:~F^k\rightarrow A$, $k=0,\ldots,T-1$ such that
\beqs
\alpha_k & = & a(Y_{0:k}^\alpha)\;,\quad k=0,\ldots,T-1\;.
\enqs 
We therefore identify in the sequel, the set $\bar{\mathfrak{C}}$ to the set $\tilde{\mathfrak{C}}$ defined by 
\beqs
\tilde{\mathfrak{C}} & := & \big\{a=(a_k)_{0\leq k\leq T-1} \text{ such that } a_k:~F^k\rightarrow A \text{ measurable}   \\
 &  &  \qquad\qquad  \text{ and } \E[|a(Y^a_{0:k})|^2]<+\infty
\text{ for } k=0,\ldots,T-1 \big\}\;.
\enqs
with $X^a=X^\alpha$ and $Y^a=Y^\alpha$ for $\alpha_k=a_k(Y_{0:k}^a)$, $k=0,\ldots,T-1$. Let us stress the well posedness of the controlled processes $X^a$ and $Y^a$ as  the components $X^a_k$ and $Y^a_k$ depend only on $\alpha_0,\ldots,\alpha_{k-1}$ for $k=1,\ldots,T$.

More precisely, we have the following identity
\beqs
V_0 & = & \sup_{a\in \tilde{\mathfrak{C}}} \tilde J(a)
\enqs
where
\beqs
\tilde J(a) & = & \E\Big[ \sum_{k=0}^{T-1}c_k\big(X_k^a,\P_{X_k^a},a_k(Y_{0:k}^a)\big)+\gamma\big(X^a_T,\P_{X^a_T}\big) \Big]~=~J(\alpha)\;.
\enqs
We next define the sets $\tilde{\mathfrak{C}}_0,\ldots,\tilde{\mathfrak{C}}_{T-1}$ by
\beqs
\tilde{\mathfrak{C}}_k(\mu) & = & \big\{ a_k:~F^{k+1}\rightarrow C \text{ measurable such that } \int_{F^{k+1}}|a_k(y_{0:k})|^2d\mu(x_{0:k},y_{0:k})~<~+\infty\big\} \;,
\enqs 
 for $\mu\in\Pc_2((E\times F)^{k+1})$ and  $k=0,\ldots,T-1$. In particular, we have
\beqs
\tilde{\mathfrak{C}} & = & 
\Big\{ a=(a_k)_{0\leq k\leq T-1} \text{ such that } 
 a_k\in \tilde{\mathfrak{C}}_k(\P_{X^a_{0:k},Y^a_{0:k}})\mbox{ for } k=0,\ldots,T-1\Big\}\;.
\enqs
We take in the rest of the paper $E=\R^p$ and $F=\R^{p'}$ and $C=\R^{p''}$ for some integers $p,p',p''\geq 1$. 
\subsection{Filtered system}
To compute the value $V_0$ defined by \eqref{defV0}, we would like to compute the conditional law of $X_k$ given $Y_{0:k}$ for $k= 1,\ldots,T$. For that, we make the following assumption.

\vspace{2mm}

\noindent \textbf{(H4)} For  $(x, y, a)\in E\times F\times{C} $ and  $k= 1,\ldots,T$, the random variable  $H_k(x, y, a, \eta_k)$ admits a density
\beqs
e & \mapsto & h_k(x,y, a, e)
\enqs
where the function $h_k$ is a $\Bc(E)\otimes\Bc(F)\otimes\Bc(C)\otimes \Bc(\R^{d'})$-measurable. 
\vspace{2mm}

Denote by $P$ the controlled transition probability of the process $X$. It is given by
\beqs
P_k^u(x_{k-1},\mu,dx_k) & = & \P(G_{k}(x_{k-1},\mu,u,\eps_{k})\in dx_{k})
\enqs
for $k\geq 1$, $x_{k-1}\in\R^d$ and $\mu\in \Pc_2(\R^d)$. 

For a given control $a\in \tilde{\mathfrak{C}}$, we get from assumption \textbf{(H4)} that the pair $(X^a_k,Y^a_k)_{0\leq k\leq T-1}$ admits the following transition 
\beq\nonumber
\P\big((X^a_k,Y^a_k)\in dx_kdy_k|X^a_{0:k-1},Y^a_{0:k-1}\big) & = & \\
h_k(x_k,Y^a_{k-1}, a_{k-1}(Y^a_{0:k-1}), y_k)P_k^{a_{k-1}(Y^a_{0:k-1})}(X^a_{k-1},\P_{X^a_{k-1}},dx_k)dy_k\label{decom-KS}  &  & 
\enq
for $k=1,\ldots,T$. 
Therefore, the joint law $\P_{(X^a_{0:k},Y^a_{0:k})}$ of $(X^a_{0:k},Y^a_{0:k})$ is given by 
\beqs
\P_{(X^a_{0:k},Y^a_{0:k})}(dx_{0:k},dy_{0:k}) & = & \P_{X_0,Y_0}(dx_0,dy_0)  \prod_{\ell=1}^k h_\ell(x_\ell, y_{\ell - 1}, a_{\ell-1}, y_{\ell}) P_\ell^{a_{\ell-1}(y_{0:\ell-1})}(x_{\ell-1},\mu_{\ell - 1}^a,dx_\ell)dy_\ell\;.
\enqs 
where $\mu^a_\ell=\P_{X_\ell^a}$ for $\ell=1,\ldots,k$.

We next introduce the measures $\Pi^a_k$, $k=0,\ldots,T$ defined by
\beqs
\Pi_{k,y_{0:k}}^a
(dx_k) & = & \P(X^a_k\in d x_k|Y^a_{0:k}=y_{0:k}) \;.
\enqs
From \eqref{decom-KS} we get the Kallianpur Streibel formula
\beqs
\Pi_{k,y_{0:k}}^a
(dx_k) & = & \frac{\int_{x_{0:k-1}}\prod_{\ell=1}^k h_\ell(x_\ell, y_{\ell - 1}, a_{\ell-1}, y_{\ell}) P_\ell^{a_{\ell-1}(y_{0:\ell-1})}(x_{\ell-1},\mu_{\ell - 1}^a,dx_\ell)\P_{X_0}(dx_0|Y_0=y_0)}{\int_{x'_{0:k}}\prod_{\ell=1}^k h_\ell(x'_\ell, y_{\ell - 1}, a_{\ell-1}, y_{\ell}) P_\ell^{a_{\ell-1}(y_{0:\ell-1})}(x'_{\ell-1},\mu_{\ell-1}^a,dx_\ell)\P_{X_0}(dx'_0|Y_0=y_0)}\;.
\enqs
We also define the unnormalized measures $\pi_{k,y_{0:k}}^a$ by
\beqs
\pi_{k,y_{0:k}}^a(dx_k) & = & \int_{x_{0:k-1}}\prod_{\ell=1}^k h_\ell(x_\ell, y_{\ell - 1}, a_{\ell-1}, y_{\ell}) P_\ell^{a_{\ell-1}(y_{0:\ell-1})}(x_{\ell-1},\mu_{\ell - 1}^a,dx_\ell)\P_{X_0}(dx_0|Y_0=y_0)\;.
\enqs
Still using \eqref{decom-KS} we get
\beqs
\Pi_{k,y_{0:k}}^a
(dx_k) & = & \frac{\pi_{k,y_{0:k}}^a
(dx_k)}{\pi_{k,y_{0:k}}^a
(\R^d)}\;.
\enqs
We now compute the probability measures $\mu_k^a$. From  \eqref{decom-KS} we have
\beqs
\mu_k^a(dx_k) & = & \int_{x_{0:k-1}}\int_{y_{0:k-1}}
P_k^{a_{k-1}(y_{0:k-1})}(x_{k-1},\mu_{k-1}^a,dx_k)h_k(x_k, y_{k - 1}, a_{k-1}, y_{k})\\
 & &  \prod_{\ell=1}^{k-1} h_\ell(x_\ell, y_{\ell - 1}, a_{\ell-1}, y_{\ell}) P_\ell^{a_{\ell-1}(y_{0:\ell-1})}(x_{\ell-1},\mu_{\ell - 1}^a,dx_\ell)dy_{1:k-1} \P_{X_0,Y_0}(dx_0,dy_0)\;.
\enqs
\subsection{The filtered problem}
We now turn to the computation of the value $V_0$. By definition we have
\beqs
V_0 & = & \inf_{a\in \tilde{\mathfrak{C}}}J(a)\\
 & =  & \inf_{a\in  \tilde{\mathfrak{C}}} \E\Big[ \sum_{k=0}^{T-1}\E\Big[ c_k\big(X_k^a,\P_{X_k^a},a_k(Y^a_{0:k})\big)|Y^a_{0:k}\Big]+\E\Big[\gamma\big(X^a_T,\P_{X^a_T}\big)|Y^a_{0:T}\Big] \Big]\;.\enqs
 Using the previous notations, we get
\beqs
V_0  & = & \inf_{a\in  \tilde{\mathfrak{C}}}\E\Big[ \sum_{k=0}^{T-1} \int c_k(x_k,\mu_k^a,a_k(Y^a_{0:k}))\Pi^a_{k,Y^a_{0:k}}(dx_k)+\int \gamma(x_T,\mu_T^a)\Pi^a_{T,Y^a_{0:T}}(dx_T)\big]\;.
\enqs
Unfortunately, this form is not time consistent. Indeed, the costs involve the marginal laws of the unobservable controlled process $X^a$ but also the controlled conditional laws $\Pi^a_{k,Y^a_{0:k}}$.  
%
\section{Extended filtered control problem}\label{Sec3}
\subsection{Extended problem}

Under the previous form the problem is not tractable as it involves the laws $\mu^a_n$ and $\Pi^a_{n,Y_{0:n}}$. {Indeed, we cannot get $\mu^a_n$ from $\Pi^a_{n,Y_{0:n}}$ and conversely}. This prevents from an application of a dynamic programming approach. To overcome this issue, we introduce the controlled measures $M_n^a\in\Pc_2\big((E\times F)^{n+1}\big)$ defined by
\beqs
M_n^a(dx_{0:n},dy_{0:n}) & = & \prod_{\ell=1}^nh_\ell(x_\ell, y_{\ell - 1}, \alpha_{\ell-1},y_\ell)P_\ell^{a_{\ell-1}(y_{0:\ell-1})}(x_{\ell-1},\mu_{\ell-1}^a,dx_\ell)dy_{1:n}\P_{X_0,Y_0}(dx_0,dy_0)
\enqs
for $n=0,\ldots,T$.
We observe that $\mu^a_n$ and $\Pi^a_{n,Y_{0:n}}$ can be computed from the measure $M_n^a$. Indeed, we first have
\beqs
\mu_n^a(dx_{0:n}) & = & \int_{y_{0:n}} M_n^a(dx_{0:n},dy_{0:n})\;.
\enqs
Secondly, we have
\beqs
\pi^a_{n,y_{0:n}}(dx_n) & = & \int_{x_{0:n-1}} {\frac{dM^a_n}{ dy_{0:n}}}(y_{0:n},dx_{0:n})\;.
\enqs
Therefore, we get
\beqs
\Pi^a_{n,y_{0:n}}(dx_n) & = & \frac{\int_{x_{0:n-1}} \frac{dM^a_n}{ dy_{0:n}}(y_{0:n},dx_{0:n})
}{\int_{x_{0:n}} \frac{dM^a_n}{dy_{0:n}}(y_{0:n},dx_{0:n})
}\;.
\enqs
We now introduce some notations. For $n\in\{0,\ldots,T\}$ and $M\in \Pc_2((E\times F)^{n+1})$ we denote by ${}_\ell M$ the $\ell$-th marginal of $M$:
\beqs
{}_\ell M(dx_\ell,dy_\ell) & = & \int_{x_0,y_0}\cdots \int_{x_{\ell-1},y_{\ell-1}}\int_{x_{\ell+1},y_{\ell+1}}\cdots \int_{x_n,y_n}M(dx_{0:n},dy_{0:n})
\enqs
 and ${}^1_\ell M$ and ${}^2_\ell M$ the first and second marginals of ${}_\ell M$ respectively:
 \beqs
{}^1_\ell M(dx_\ell) & = & \int_{y_\ell}{}_\ell M(dx_\ell,dy_\ell)\\
{}^2_\ell M(dy_\ell) & = & \int_{x_\ell}{}_\ell M(dx_\ell,dy_\ell)
 \enqs 
 for $\ell\in\{0,\ldots,n\}$. 
We define the controlled transition probability $\bar P_{n}:(E\times F )^{n}\times \tilde{\mathfrak{C}} \times \Pc_2((E\times F)^{n})\rightarrow\Pc_2(E\times F)$ by
\beqs
\bar P^{a_{n-1}}_{n} (z,M,dw) & = & P^{a_{n-1}(y_{0:n-1})}_{n}(x_{n-1},{}_{n-1}^1M,dw_{1})h_n(w_1,y_{n-1}, a_{n-1}(y_{0:n-1}), w_2)dw_2
\enqs 
for $z=(x_{\ell},y_{\ell})_{0\leq\ell\leq n-1}\in (E\times F)^{n}$, $M\in \Pc_2\big((E\times F)^{n})$, $a_{n-1}\in \tilde{\mathfrak{C}}_{n-1}(M)$ and $w=(w_1,w_2)\in E\times F$ and  $n=1,\ldots,T$. 
The controlled measures $\P_{X_n^a,Y_n^a}$ have the following dynamics
\beqs
\P_{(X_{0:n+1}^a,Y_{0:n+1}^a)}(dx_{0:n +1},dy_{0:n +1}) & = & \\
 \bar P^{a_n}_{n+1}(x_{0:n},y_{0:n}, a_n(Y_{0:n}), \P_{(X_{0:n}^a,Y_{0:n}^a)},dx_{n+1},dy_{n+1})\P_{(X_{0:n}^a,Y_{0:n}^a)}(dx_{0:n},dy_{0:n}) & & 
\enqs
for $n=0,\ldots,T-1$.
In the sequel, we use the following notation : for $\mu\in \Pc_2((E\times F)^{n+1})$ and $a_n\in \tilde{\mathfrak{C}}_n(\mu)$, we define  $\bar P^{a_n}_{n+1}\mu \in \Pc_2((E\times F)^{n+2})$ by
\beqs
 \bar P^{a_n}_{n+1}\mu (dx_{0:n+1},dy_{0:n+1}) & = & \bar P^{a_n}_{n+1}(x_{0:n},y_{0:n}, \mu,dx_{n+1},dy_{n+1})\mu(dx_{0:n},dy_{0:n})
\enqs
for $n=0,\ldots,T-1$. The dynamics of the controlled measures $\P_{(X_{0:n}^a,Y_{0:n}^a)}$ can be rewritten under the following simplified form
\beqs
\P_{(X_{0:n+1}^a,Y_{0:n+1}^a)}(dx_{0:n+1},dy_{0:n+1}) & = & \bar P^{a_n}_{n+1}\P_{(X_{0:n}^a,Y_{0:n}^a)}(dx_{0:n},dy_{0:n})
\enqs
for $n=0,\ldots,T-1$.

We now turn to the filtered problem. We define the cost functions $C_n$, $n=0,\ldots,T-1$ and $\Gamma$ by
\beqs
C_n(M,a_n) & = & \int c_n\Big(x_n,{}^1_nM,a_n(y_{0:n})\Big)dM(x_{0:n},y_{0:n})
\enqs
for $M\in\Pc_2\big((E\times F)^{n+1})\big)$, $a_n\in\tilde{\mathfrak{C}}_n(M)$
and 
\beqs
\Gamma (M) &  = & \int \gamma\Big(x_T,{}^1_TM\Big)dM(x_{0:T},y_{0:T})\;,
\enqs
for  $M\in\Pc_2\big((E\times F)^{T+1}\big)$.
A straightforward computation gives 
\beqs
V_0 & = & \inf_{a\in \tilde{\mathfrak{C}}}\bar J(a)\;.
\enqs
where the criteria $\bar J$ is defined by
\beq\label{defJbar}
\bar J(a) & := &  \sum_{n=0}^{T-1}C_k\big(\P_{(X_{0:n}^a,Y_{0:n}^a)},a_n\big)+\Gamma\big(\P_{(X_{0:T}^a,Y_{0:T}^a)}\big)\;,\quad a\in\tilde{\mathfrak{C}} \;.
\enq

\subsection{Dynamic programming}
We dynamically extend the value $ V_0$. For that, we define for $n=0,\ldots,T$ the functions $V_n: \Pc_2((E\times F)^{n+1})\rightarrow \R$ by
\beqs
V_n(\mu) & = & \inf_{a\in \;{}_n\tilde{\mathfrak{C}}(\mu)} \bar J_n(\mu,a)
\enqs 
 where
\beqs
\bar J_n(\mu,a)  & := &  \sum_{k=n}^{T-1}C_k\big(M^{n,\mu,a}_k,a_k\big)+\Gamma\big(M^{n,\mu,a}_T\big) 
\enqs
with $(M^{n,\mu,a}_k)_{k=n,\ldots,T}$  defined by $M^{n,\mu,a}_n=\mu$ and 
\beq\label{dynM}
M^{n,\mu,a}_{k+1} & = & \bar P^{a}_{k+1}M^{n,\mu,a}_{k}  
\enq
for $k=n,\ldots,T-1$, and
\beqs
{}_n\tilde{\mathfrak{C}}(\mu) & = & \Big\{ a=(a_k)_{0\leq k\leq T-1} \text{ such that } 
 a_k\in \tilde{\mathfrak{C}}_k(M^{n,\mu,a}_{k})\mbox{ for } k=n,\ldots,T-1\Big\}\
\enqs
for any $\mu\in \Pc_2((E\times F)^{n+1})$

\begin{Lemma}[Dynamic programming]\label{dynprog}
The value functions $V_0,\ldots,V_T$ satisfy the following dynmic programming principle
\beq\label{Dyn-prog}
V_n(\mu) & = & \inf_{a_n\in\tilde{\mathfrak{C}}_n(\mu)} \big\{C_n\big(\mu,a_n\big)+V_{n+1}\big( \bar P^{a_n}_{n+1}\mu\big)\big\}\;,\quad \mu\in \Pc_2((E\times F)^{n+1})\;,\enq
for $n=0,\ldots,T-1$
\end{Lemma}
\begin{proof}Denote by $Z_n(\mu)$ the right-hand-side of \eqref{Dyn-prog}. Fix a strategy $a\in{}_{n}\tilde{\mathfrak{C}}(\mu)$. We then have
$ a\in{}_{n+1}\tilde{\mathfrak{C}}(\bar P^{a_n}_n\mu)$ .
We then notice that
\beqs
M^{n+1,\bar P^{a_n}_{n+1}\mu,a}_k & = & M^{n,\mu,a}_k
\enqs
for  $k=n+1,\ldots,T$. Therefore, we have
\beqs
V_{n+1}( \bar P^{a_n}_{n+1}\mu)
 & \leq & \sum_{k=n+1}^{T-1}C_k\big(M^{n,\mu,a}_k,a_k\big)+\Gamma\big(M^{n,\mu,a}_T\big) 
\enqs
and
\beqs
C_n\big(\mu,a_n\big)+V_{n+1}( \bar P^{a_n}_{n+1}\mu) & \leq & \bar J_n(\mu,a).
\enqs
Taking the infimum over $a\in{}_{n}\tilde{\mathfrak{C}}(\mu)$, we get
\beqs
Z_n(\mu) & \leq & V_n(\mu)\;.
\enqs
We turn to the reverse inequality. Fix $\eps>0$ and a strategy $a^\eps_n\in\tilde{\mathfrak{C}}_n(\mu)$ such that
\beq\label{cond1PPDleq}
C_n\big(\mu,a^\eps_n\big)+V_{n+1}( \bar P^{a_n^\eps}_{n+1}\mu
)-{\frac{\eps}{ 2}} & \leq & Z_n(\mu)\;.
\enq
We now fix a strategy $\tilde a^\eps\in {}_{n+1}\tilde{\mathfrak{C}}(\bar P^{a_n^\eps}_{n+1}\mu)$ such that
\beq\label{cond2PPDleq}
\bar J_{n+1}( \bar P^{a_n^\eps}_{n+1}\mu, 
\tilde a ^\eps) - \frac{\eps}{ 2} & \leq & V_{n+1}( \bar P^{a_n^\eps}_{n+1}\mu)
\;.
\enq
We define $\hat a ^\eps$ as the concatenation of $a^\eps_n$ and  $\tilde a^\eps$:
\beqs
\hat a ^\eps_k(.) & = & a^\eps_n(.)\mathds{1}_{k=n}+\tilde a_k^\eps(.)\mathds{1}_{k\neq n}\;.
\enqs
 Then we have $\hat a \in {}_n\tilde{\mathfrak{C}}(\mu)$, $M^{n,\mu,\hat a^\eps}_{n+1}  =   \bar P^{a_n^\eps}_{n+1}\mu
$ 
and
\beqs
M^{n,\mu,\hat a^\eps}_{k} & = & M^{n+1,\bar P^{a_n^\eps}_{n+1}\mu,\tilde a^\eps}_{k}
\enqs
for $k=n+1,\ldots,T$. Therefore we get from \eqref{cond1PPDleq} and \eqref{cond2PPDleq}
\beqs
\bar J_n(\mu, \hat a ^\eps) -\eps & \leq & Z_n(\mu)
\enqs
and 
\beqs
V_n(\mu) -\eps & \leq & Z_n(\mu)\;.
\enqs
Since $\eps$ is arbitrarily chosen, we get  $V_n(\mu) \leq  Z_n(\mu)$.
\end{proof}
We now provide a verification result for the optimal value and an optimal strategy.
\begin{Theorem}[Verification]\label{VerifTHM}
Consider the functions $W_0,\ldots,W_T$ defined by
\beqs
W_T(\mu) & = & \Gamma (\mu)\;,\quad \mu\in \Pc_2((E\times F)^{T+1})\;,
\enqs
and 
\beq\label{DynVerif}
W_n(\mu) & = & \inf_{a_n\in\tilde{\mathfrak{C}}_n(\mu)} \big\{C_n\big(\mu,a_n\big)+W_{n+1}( \bar P^{a_n}_{n+1}\mu)\big\}\;,\quad \mu\in \Pc_2((E\times F)^{n+1})
\enq
 for $ n=0,\ldots, T-1$. Then $W_n=V_n$ for $n=0,\ldots,T$. 
 
 Suppose that for  any $n=0,\ldots,T-1$ and any $\mu\in\Pc_2((E\times F)^n)$, there exist  a function $a_n^{\mu,*}\in \tilde{\mathfrak{C}}_n(\mu)$ such that
\beq\label{strat-opt}
W_n(\mu) & = &  C_n\big(\mu,a^{\mu,*}_n\big)+W_{n+1}( \bar P^{a^{\mu,*}_{n+1}}_n\mu). 
\enq 
Then for a given starting measure $\mu\in\Pc_2(E\times F)$, the strategy $a^{*}$ defined by $a^{*}_0=a^{*,\mu}_0$ and
\beq\label{def-strat-opt}
a^{*}_n & = & a^{M_{n}^{0,\mu,a^*},*}_n\;,~n=1,\ldots,T-1\;,
\enq
is optimal: 
\beqs
V_0(\mu) & = & \bar J(\mu,a^*)\;.
\enqs
\end{Theorem}
\begin{proof}Fix $n=0,\ldots,T-1$, $\mu\in\Pc_2((E\times F)^{n+1})$ and a control $a\in {}_n\tilde{\mathfrak{C}}(\mu)$. Since $W_T=\Gamma$, we get by \eqref{DynVerif} and a straightforward backward induction 
\beqs
W_n(\mu) & \leq & \sum_{k=n}^{T-1}C_k\big(M^{n,\mu,a}_k,a_k\big)+\Gamma\big(M^{n,\mu,a}_T\big) \;.
\enqs
Since $a\in {}_n\tilde{\mathfrak{C}}(\mu)$ is arbitrarily chosen, we get
\beqs
W_n(\mu) & \leq & V_n(\mu) 
\enqs
for $\mu\in\Pc_2((E\times F)^{n+1})$ and $n=0,\ldots,T$.

We now prove the reverse inequality by a backward induction. First we have $V_T=\Gamma=W_T$, so $W_T\geq V_T$.

Suppose that we have $W_{n+1}\geq V_{n+1}$. Using Lemma \ref{dynprog} and  \eqref{DynVerif} we get $W_n\geq V_n$. Therefore,  the result holds true for any $n=0,\ldots,T$.

Fix now $\mu\in\Pc_2(E\times F)$ and consider the strategy $a^{*}$ defined by \eqref{def-strat-opt}.
Writing the equality \eqref{DynVerif} at each step we get
\beqs
W_0(\mu) & = & \hat J(\mu,a^*)
\enqs
and $a^*$ is optimal since $W_0(\mu)=V_0(\mu)$.
\end{proof}
In the previous result, the verification condition \eqref{DynVerif} is written over the set $\tilde{\mathfrak{C}}_n$ which might be too large.   Indeed, as the dynamics \eqref{dynM}  of $M^{n,\mu,a}$ involves the past only through the control $a$, one may wonder if condition \eqref{DynVerif} can be reduced to closed loop controls, $i.e.$ controls depending on the the present value of $M^{n,\mu,a}$, and if an optimal strategy of this form can be derived. This is possible under the condition that the starting position measure $\mu$ already has this structure.

More precisely, we introduce,   for $n=0,\ldots,T-1$ and $\mu\in \Pc_2((E\times F)^n)$ the subset $\tilde{\mathfrak{C}}_{n,\textrm{cl}}(\mu)$ of $\tilde{\mathfrak{C}}_n(\mu)$ composed by  control functions $a_n$ depending only on the last component 
\beqs
a_n (y_{0:n}) & = & a_n(y_{n})\;,\quad y_{0:n}\in F^{n+1}\;.
\enqs 
 We then define for $\mu\in \Pc_2(E\times F)$ the set  $\tilde{\mathfrak{C}}_{\textrm{cl}}(\mu)$ as the set of controls $a=(a_0,\ldots,a_{T-1})\in{}_0\tilde{\mathfrak{C}}(\mu)$ such that $a_n\in  \tilde{\mathfrak{C}}_{n,\textrm{cl}}(M^{0,\mu,a}_n)$ for $n=0,\ldots,T-1$. 
We finally denote by $\Pc_{\textrm{cl}}\big((E\times F)^n\big)$ the subset of $\Pc_2\big((E\times F)^n\big)$ composed by laws of controlled processes $(X^a_{0:n},Y^a_{0:n})$ for $a\in \tilde{\mathfrak{C}}_{\textrm{cl}}(\mu)$ with $\P_{(X_0^a,Y_0^a)}=\mu$ and $\mu\in\Pc_2(E\times F)$. We notice that $\Pc_{\textrm{cl}}\big(E\times F\big)=\Pc_{2}\big(E\times F\big)$ as the initial position does not depend on the control. We can now state the verification result for closed loop controls.  

\begin{Theorem}[Verification with closed loop controls]\label{VerifTHMCL}
Consider the functions $W_0,\ldots,W_T$ defined by
\beqs
W_T(\mu) & = & \Gamma (\mu)\;,\quad \mu\in \Pc_{\textrm{cl}}((E\times F)^{T+1})\;,
\enqs
and 
\beq\label{DynVerifCL}
W_n(\mu) & = & \inf_{a_n\in\tilde{\mathfrak{C}}_{n}(\mu)} \big\{C_n\big(\mu,a_n\big)+W_{n+1}( \bar P^{a_n}_{n+1}\mu)\big\}\;,\quad \mu\in \Pc_{\textrm{cl}}((E\times F)^{n+1})\;,
\enq
 for $ 0\leq n\leq T-1$. Then $W_n=V_n$ for $n=0,\ldots,T$ on $\Pc_{\textrm{cl}}((E\times F)^{n+1})$. 
 Suppose that for  any $n=0,\ldots,T-1$ there exists a measurable function $\mathfrak{a}_n:~\Pc_2((E\times F)^n)\times F\rightarrow  C$ such that
\beq\label{strat-optCL}
W_n(\mu) & = &  C_n\big(\mu,\mathfrak{a}_n(\mu,.)\big)+W_{n+1}( \bar P^{\mathfrak{a}_n(\mu,.)}_{n+1}\mu). 
\enq 
and $\mathfrak{a}_n(\mu,.)\in\tilde{\mathfrak{C}}_{n}$ for all $\mu\in\Pc_{\textrm{cl}}((E\times F)^{n})$.
Then for a given starting measure $\mu\in\Pc_2(E\times F)$, the strategy $a^{*}$ defined by $a^{*}_0=\mathfrak{a}_0(\mu,.)$ and
\beqs
a^{*}_n & = & \mathfrak{a}_n(M_{n}^{0,\mu,a^*},.)
\;,~n=1,\ldots,T-1\;,
\enqs
is optimal: 
\beqs
V_0(\mu) & = & \bar J(\mu,a^*)\;.
\enqs
\end{Theorem}
\begin{proof} The proof follows exactly the same lines as that of Theorem \ref{VerifTHM} and is therefore omitted.
%
%
%
\end{proof}

\section{Applications}\label{Sec4}
\subsection{Linear-quadratic case} We take $p=p'=p''=d'=d$ in this section. 
We suppose that the processes $X$ and $Y$ have the following dynamics
\beq\label{dynX}
X_{k+1}^a & = & B_kX^a_k+\bar B_k \E[X^a_k]+D_k a_k(Y^a_{0:k})+\eps_{k+1}\\\label{dynY}
Y_{k+1}^a & = & J_{k+1}X^a_{k+1}+\eta_{k+1}\;,\label{dynYLQ}
\enq
where $B_k$, $\bar B_k$ $D_k$ and $J_k$ are deterministic $d\times d$ matrices and  $\eps_k$ and $\eta_k$ follow $\Nc(0,I_d)$. We also suppose that $X_0$ and $Y_0$ are independent and follow $\Nc(0,I_d)$. We then define the cost function $\tilde J$ by 
\beqs
\tilde J(a) & := & \E\Big[ \sum_{k=0}^{T-1}{X_k^a}^\top Q_kX_k^a+\E[X_k^a]^\top \bar Q_k\E[X_k^a]
+a(Y^a_{0:k})^\top R_ka(Y^a_{0:k})\\
 & & +{X_T^a}^\top Q_TX_T^a+\E[X_T^a]^\top \bar Q_T\E[X_T^a]
  \Big]\;,
\enqs
for any strategy $a\in\tilde{\mathfrak{C}}$. In this case  \textbf{(H1)}-\textbf{(H2)}-\textbf{(H3)} are satisfied and we have
\beqs
h_k(x,y,a,e) ~~=~~h_k(x,e) & = & \frac{1}{ (\sqrt{2\pi})^d}\exp\Big(-\frac{1}{2}\big|e-J_{k}x\big|^2\Big)\;,\quad x,y,a,e\in \R^d\;, 
\enqs
which is the density of the law $\Nc(J_{k}x,I_d)$.  We now introduce some notations. 
For $\mu\in\Pc_2(E\times F)$ and $\Lambda\in \R^{2d\times 2d}$, we set 
\beqs
\bar \mu & = & \int_{\R^{2d} } \left(\begin{array}{c}x\\y\end{array}\right) d\mu(x,y)
\enqs
 and 
\beqs
\langle \mu\rangle (\Lambda) & = & \int_{\R^{2d} }\left(\begin{array}{c}x\\y\end{array}\right)^\top\Lambda \left(\begin{array}{c}x\\y\end{array}\right)d\mu(x,y)\;.
%
\enqs
We also define the matrices $\mathbf{Q}_k, \bar{\mathbf{Q}}_k\in\R^{2d\times2d}$ by
\beqs
\mathbf{Q}_k  =  \left(\begin{array}{cc}Q_k & 0\\ 0 & 0 \end{array}\right)  & \mbox{ and } &\bar{\mathbf{Q}}_k  =  \left(\begin{array}{cc}\bar Q_k & 0\\ 0 & 0 \end{array}\right) 
\enqs
for $k=0,\ldots,T$.
The functions $C_k$ and $\Gamma$ appearing in the definition \eqref{defJbar} of $\bar J$ are then given by
\beqs
C_k(\mu,a) & = & \langle {}_k\mu\rangle(\mathbf{Q}_k)+{}_k\bar{\mu}^\top\mathbf{Q}_k{}_k\bar{\mu}
+\int a(y_{0:k})^\top R_ka(y_{0:k})d\mu(x_{0:k},y_{0:k})
\enqs
for  $\mu\in \Pc_2((E\times F)^{k+1})\;,~a\in \tilde{\mathfrak{C}}_k(\mu)$, $k=0,\ldots,T-1$,  and 
\beqs
\Gamma(\mu) & = & \langle {}_T\mu\rangle (\mathbf{Q}_T)+{}_T\bar{\mu}^\top\big(\mathbf{Q}_T\big){}_T\bar{\mu}\;,\quad \mu\in \Pc_2((E\times F)^{T+1})\;,
\enqs
where we recall that  ${}_k\mu$ stands for the $k$-th marginal of $\mu$ for $k=0,\ldots,T$. 

We look for candidates $W_k$, $k = 0,\ldots,T$,  satisfying the verification Theorem. For that we chose an ansatz in the following quadratic form:
\beq\label{ansatz}
W_k(\mu) & = & \langle {}_k\mu\rangle ( \Lambda_k ) +  {}_k\bar \mu^\top \Theta_k {}_k\bar \mu +
 \chi_k
\enq
for $\mu\in \Pc_2((E\times F)^{k+1})$. 
We next suppose that $J_1,\ldots,J_T$  and $D_0,\ldots,D_{T-1}$ are all invertible and that $Q_1,Q_1+\bar Q_1,\ldots,Q_T,Q_T+\bar Q_T$ are symmetric nonnegative and $R_1,\ldots,R_T$ are all symmetric positive.
We then have the following result.
\begin{Proposition} \label{kalman_bucy_proposition}
There exists  $\Lambda_k,\Theta_k\in\R^{2d\times 2d}$ symmetric, with $\Lambda_k$ and $\Theta_k+\Lambda_k$ nonnegative,  and $\chi_k\in\R$, $k=0,\ldots,T$, symmetric nonnegative such that 
 the functions $W_k$, $k=0,\ldots,T$, given  by \eqref{ansatz}  satisfy the verification Theorem \ref{VerifTHMCL} 
with a feedback optimal strategy $\mathfrak{a}$ of the form
\beqs
\mathfrak{a}_n(\mu,y) & = & G_n\Big(\Xi_n \Phi_n(\mu,y)+ \hat{\mathbf{S}}_n^\top 
    {}_{n}\bar{\mu}\Big)\;,\quad y\in \R^d\;,~\mu\in \Pc_{\textrm{cl}}((\R^d\times\R^d)^{n+1})\;,
\enqs
with $\Phi_n$ given by 
\beq\label{defPhi}
\Phi_n(\mu,y) & = & \int_{\R^d} x_n \frac{e^{-\frac{1}{2}|J_nx_n-y_n|^2}}{\int_{\R^d}e^{-\frac{1}{2}|J_nx'_n-y_n|^2}d\;{}^1_n\mu(x'_n)}d\;{}^1_n\mu(x_n),
\enq
and $G_n\in\R^{d\times d}$, $\Xi_n\in\R^{d\times d}$ and $\hat{\mathbf{S}}_n\in\R^{2d\times d}$ 
 for $n=0,\ldots,T-1$

\end{Proposition}
\begin{proof}
We use a backward induction on $k$ to prove the following statement:

\vspace{2mm}

For $n=0,\ldots,N$, there exists $\Lambda_k,\Theta_k\in\R^{2d\times 2d}$ symmetric, with $\Lambda_k$ and $\Theta_k+\Lambda_k$ nonnegative, and $\chi_k\in\R$, $k=0,\ldots,T$, such that 
the functions $W_k$, $k=n,\ldots,T$, given  by \eqref{ansatz}  satisfy the verification Theorem \ref{VerifTHMCL} with a feedback optimal strategy $\mathfrak{a}$  of the form
\beqs
\mathfrak{a}_k(\mu,y) & = & G_k\Big(\Xi_k \Phi_k(\mu,y)+ \hat{\mathbf{S}}_k^\top 
    {}_{k}\bar{\mu}\Big)\;,\quad y\in \R^d\;,~\mu\in \Pc_{\textrm{cl}}((\R^d\times\R^d)^{k+1})\;,
\enqs
with $G_k\in\R^{d\times d}$, $\Xi_k\in\R^{d\times d}$ and $\hat{\mathbf{S}}_k\in\R^{2d\times d}$ 
 some matrices depending on the coefficients for $k=n,\ldots,T-1$ 


\vspace{2mm}

\noindent For $n=T$ a straightforward computation gives 
\beqs
\Lambda_T~=~\mathbf{Q}_T\,,~\Theta_T~=~\bar{\mathbf{Q}}_T  & \mbox{ and } & \chi_T~=~0\;. 
\enqs
Therefore, the property holds for $n=T$. 

Suppose that the property holds for $n+1$. Fix $\mu\in \Pc_{\textrm{cl}}((E\times F)^{n+1})$. From the induction assumption, we have
\beqs
\inf_{a_n\in\tilde{\mathfrak{C}}_{n}(\mu)} \big\{C_n\big(\mu,a_n\big)+W_{n+1}( M^{n,\mu,a_n}_{n+1})\big\} & = & \\
\inf_{a_n\in\tilde{\mathfrak{C}}_n(\mu)} \big\{C_n\big(\mu,a_n\big)+
\langle{}_{n+1}M^{n,\mu,a_n}_{n+1}\rangle ( \Lambda_{n+1} )  \\
+  ({}_{n+1}\bar{M}^{n,\mu,a}_{n+1})^\top \Theta_{n+1} ({}_{n+1}\bar M^{n,\mu,a_n}_{n+1}) 
 + \chi_{n+1}
\big\} & = & \\
\inf_{a_n\in\tilde{\mathfrak{C}}_n(\mu)} \big\{\langle{}_{n}\mu\rangle(\mathbf{Q}_n)+{}_{n}\bar{\mu}^\top\bar{\mathbf{Q}}_n{}_{n}\bar{\mu} & & \\
+\int a_n(y_{0:n})^\top R_na_n(y_{0:n})d\mu(x_{0:n},y_{0:n})+ \langle{}_{n+1}M^{n,\mu,a_n}_{n+1}\rangle ( \Lambda_{n+1} ) & & \\
+  ({}_{n+1}\bar{M}^{n,\mu,a}_{n+1})^\top \Theta_{n+1} ({}_{n+1}\bar M^{n,\mu,a_n}_{n+1}) 
+ \chi_{n+1}
\big\}\;.
\enqs
From \eqref{dynM} we have 
\beqs
{}_{n+1}\bar M^{n,\mu,a_n}_{n+1} & = & \int_{\R^{2d} }\left(\begin{array}{c}x_{n+1}\\y_{n+1}\end{array}\right)d \;{}_{n+1}M^{n,\mu,a_n}_{n+1}(x_{n+1},y_{n+1})\\
   & = &  \int_{\R^{2d} }\left(\begin{array}{c}x_{n+1}\\y_{n+1}\end{array}\right)\int_{(\R^{2d})^n }P^{a_n(y_{0:n})}_n(x_{n},{}_{n}^1\mu,dx_{n+1})h_n(x_{n+1},y_{n+1})dy_{n+1}d\mu(x_{0:n},y_{0:n})
  \enqs
  From \eqref{dynX}-\eqref{dynY} we get
  \beqs
{}_{n+1}\bar M^{n,\mu,a_n}_{n+1}   & = &  \frac{1}{(2\pi)^d}\int_{\R^{2d} }\left(\begin{array}{c}x_{n+1}\\y_{n+1}\end{array}\right)\int_{(\R^{2d})^n }e^{-\frac{1}{2}|x_{n+1}-(B_nx_n+\bar{B}_n {}_{n}^1\bar{\mu}+D_na_n(y_{0:n}))|^2}\\
 & & \qquad\qquad\qquad\qquad\qquad e^{-\frac{1}{2}|y_{n+1}-J_{n+1}x_{n+1}|^2}d\mu(x_{0:n},y_{0:n})dx_{n+1}dy_{n+1}\\
 & = & \int_{(\R^{2d})^n }\left(\begin{array}{c}(B_nx_n+\bar{B}_n {}_{n}^1\bar{\mu}+D_na_n(y_{0:n}))\\J_{n+1}(B_nx_n+\bar{B}_n {}_{n}^1\bar{\mu}+D_na_n(y_{0:n}))\end{array}\right)d\mu(x_{0:n},y_{0:n})\\
  & = & \mathbf{J}_{n+1}(B_n+\bar{B}_n) {}_{n}^1\bar{\mu}+  \mathbf{J}_{n+1}D_n\int_{(\R^{2d})^n}a_n(y_{0:n})d\mu(x_{0:n},y_{0:n})
\enqs
where 
\beqs
\mathbf{J}_{n+1} & = & \left(\begin{array}{c}I_d\\J_{n+1}\end{array}\right)\in \R^{2d\times d}\;.
\enqs
Therefore we get
\beqs
({}_{n+1}\bar  M^{n,\mu,a_n}_{n+1})^\top \Theta_{n+1} ({}_{n+1}\bar  M^{n,\mu,a_n}_{n+1}) & = & \\
{}_{n}\bar{\mu}^\top (\mathbf{B}_n+\bar{\mathbf{B}}_n)^\top\mathbf{J}_{n+1}^\top\Theta_{n+1} \mathbf{J}_{n+1}(\mathbf{B}_n+\bar{\mathbf{B}}_n){}_{n}^1\bar{\mu} &  & \\
 +2{}_{n}\bar{\mu}^\top (\mathbf{B}_n+\bar{\mathbf{B}}_n)^\top\mathbf{J}_{n+1}^\top\Theta_{n+1} \mathbf{J}_{n+1}D_n\int_{(\R^{2d})^n}a_n(y_{0:n})d\mu(x_{0:n},y_{0:n})& &\\
   +\Big(\int_{(\R^{2d})^n}a_n(y_{0:n})d\mu(x_{0:n},y_{0:n})\Big)^\top D_n^\top\mathbf{J}_{n+1}^\top\Theta_{n+1} \mathbf{J}_{n+1}D_n\int_{(\R^{2d})^n}a_n(y_{0:n})d\mu(x_{0:n},y_{0:n}) & &
\enqs
with
\beqs
\mathbf{B}_n ~ = ~ \left(\begin{array}{cc}B_n  &  0\end{array}\right)\in \R^{ d\times 2d}  & \mbox{and} & 
\bar{\mathbf{B}}_n ~ = ~ \left(\begin{array}{cc}\bar B _n & 0\end{array}\right)\in \R^{  d\times 2d} \;.
\enqs
We turn to the computation of the second order moment. 
  Still using \eqref{dynX}-\eqref{dynY}, a computation gives
  \beqs
\int_{\R^{2d} }\left(\begin{array}{c}x_{n+1}\\y_{n+1}\end{array}\right)^\top\Lambda_{n+1} \left(\begin{array}{c}x_{n+1}\\y_{n+1}\end{array}\right)d({}_{n+1}  M^{n,\mu,a_n}_{n+1})(x_{n+1},y_{n+1})  & = &  \\\int_{(\R^{2d})^n }{\frac{1}{ (2\pi)^d}}\int_{\R^{2d} }\left(\begin{array}{c}x_{n+1}\\y_{n+1}\end{array}\right)^\top\Lambda_{n+1} \left(\begin{array}{c}x_{n+1}\\y_{n+1}\end{array}\right)e^{-\frac{1}{2}|x_{n+1}-(B_nx_n+\bar{B}_n {}_{n}^1\bar{\mu}+D_na_n(y_{0:n}))|^2}\\
 e^{-\frac{1}{2}|y_{n+1}-J_{n+1}x_{n+1}|^2}dx_{n+1}dy_{n+1}d\mu(x_{0:n},y_{0:n}) & = & \\
\int_{(\R^{2d})^n} (B_nx_n+\bar{B}_n {}_{n}^1\bar{\mu}+D_na_n(y_{0:n}))^\top\mathbf{J}_{n+1} ^\top\Lambda_{n+1}\mathbf{J}_{n+1}(B_nx_n+\bar{B}_n {}_{n}^1\bar{\mu}+D_na_n(y_{0:n}))d\mu(x_{0:n},y_{0:n})\\
+ \textrm{Tr}(\mathbf{J}_{n+1}^\top\Lambda_{n+1}\mathbf{J}_{n+1})+ \textrm{Tr}(\mathbf{I}^\top \Lambda_{n+1}\mathbf{I})
\enqs
where 
\beqs
\mathbf{I}_{} & = & \left(\begin{array}{c}0 \\ I_d\end{array}\right)\in \R^{2d\times d}\;.
\enqs
We therefore finally get
\beqs
\langle{}_{n+1}M^{n,\mu,a_n}_{n+1}\rangle ( \Lambda_{n+1} ) & = & \\
\langle{}_{n}\mu\rangle ( \mathbf{B}_n^\top\mathbf{J}_{n+1}^\top\Lambda_{n+1}\mathbf{J}_{n+1}\mathbf{B}_n) & & \\
+{}_{n}\bar\mu^\top ( \bar{\mathbf{B}}_n^\top\mathbf{J}_{n+1}^\top\Lambda_{n+1}\mathbf{J}_{n+1}\bar{\mathbf{B}}_n+{\mathbf{B}}_n^\top\mathbf{J}_{n+1}^\top\Lambda_{n+1}\mathbf{J}_{n+1}\bar{\mathbf{B}}_n+\bar{\mathbf{B}}_n^\top\mathbf{J}_{n+1}^\top\Lambda_{n+1}\mathbf{J}_{n+1}{\mathbf{B}}_n ){}_{n}\bar\mu & & \\
+ \textrm{Tr}(\mathbf{J}_{n+1}^\top\Lambda_{n+1}\mathbf{J}_{n+1})+ \textrm{Tr}(\mathbf{I}^\top \Lambda_{n+1}\mathbf{I})\\
+ \int_{(\R^{2d})^n} a_n(y_{0:n})^\top {D}_n^\top \mathbf{J}_{n+1}^\top\Lambda_{n+1}\mathbf{J}_{n+1} D_na_n(y_{0:n})d\mu(x_{0:n},y_{0:n}) & & \\
+ 2\int_{(\R^{2d})^n} x_n^\top {B}_n^\top\mathbf{J}_{n+1}^\top\Lambda_{n+1}\mathbf{J}_{n+1} D_na_n(y_{0:n})d\mu(x_{0:n},y_{0:n}) & & \\
+2{}_n\bar \mu^\top\int_{(\R^{2d})^n}  \bar {\mathbf{B}}_n^\top \mathbf{J}_{n+1}^\top\Lambda_{n+1}\mathbf{J}_{n+1} D_na_n(y_{0:n})d\mu(x_{0:n},y_{0:n})\;. & & 
\enqs
We now go back to the definition of $W_n$:
\beqs
W_n(\mu)  & = &  \inf_{a_n\in\tilde{\mathfrak{C}}_n(\mu)} \Big\{\langle{}_{n}\mu\rangle(\mathbf{Q}_n)+{}_{n}\bar{\mu}^\top\bar{\mathbf{Q}}_n{}_{n}\bar{\mu}+\int a_n(y_{0:n})^\top R_na_n(y_{0:n})d\mu(x_{0:n},y_{0:n})
 \\
  & & \qquad \qquad+\langle{}_{n+1}M^{n,\mu,a_n}_{n+1}\rangle ( \Lambda_{n+1} )
 +  ({}_{n+1}\bar{M}^{n,\mu,a}_{n+1})^\top \Theta_{n+1}( {}_{n+1}\bar M^{n,\mu,a_n}_{n+1}) + \chi_{n+1}
\Big\} \\
 & = & \inf_{a\in\tilde{\mathfrak{C}}_n(\mu)} \Big\{\langle{}_{n}\mu\rangle(\mathbf{Q}_n+\mathbf{B}_n^\top\mathbf{J}_{n+1}^\top\Lambda_{n+1}\mathbf{J}_{n+1}\mathbf{B}_n)  +({}_{n}\bar{\mu})^\top \mathbf{K}_n
  ({}_{n}\bar{\mu})\\
   & & \qquad \qquad+ \textrm{Tr}(\mathbf{J}_{n+1}^\top\Lambda_{n+1}\mathbf{J}_{n+1})+\textrm{Tr}(\mathbf{I}^\top \Lambda_{n+1}\mathbf{I})+\chi_{n+1}\\
   & &  \qquad \qquad+ \Big(\int_{(\R^{2d})^n}a_n(y_{0:n})d\mu(x_{0:n},y_{0:n})\Big)^\top 
   \mathbf{N}_n\int_{(\R^{2d})^n}a_n(y_{0:n})d\mu(x_{0:n},y_{0:n})\\
    & & \qquad \qquad+ 2{}_{n}\bar{\mu}^\top 
    \mathbf{S}_n    \int_{(\R^{2d})^n} a_n(y_{0:n})d\mu(x_{0:n},y_{0:n})  \\
     & & \qquad \qquad+ 2\int_{(\R^{2d})^n} x_n^{\top} {B}_n^\top\mathbf{J}_{n+1}^\top\Lambda_{n+1}\mathbf{J}_{n+1} D_na_n(y_{0:n})d\mu(x_{0:n},y_{0:n}) \\
 & & \qquad \qquad+ \int_{(\R^{2d})^n} a_n(y_{0:n})^\top \big({D}_n^\top\mathbf{J}_{n+1}^\top\Lambda_{n+1}\mathbf{J}_{n+1} D_n+ R_n\big)a_n(y_{0:n})d\mu(x_{0:n},y_{0:n})\Big\}
\enqs
where
\beqs
\mathbf{K}_n & = &  \bar{\mathbf{Q}}_n+ (\mathbf{B}_n+\bar{\mathbf{B}}_n)^\top\mathbf{J}_{n+1}^\top\Theta_{n+1}\mathbf{J}_{n+1}(\mathbf{B}_n+\bar{\mathbf{B}}_n)+\bar{\mathbf{B}}_n^\top\mathbf{J}_{n+1}^\top\Lambda_{n+1}\mathbf{J}_{n+1}\bar{\mathbf{B}}_n\;,\\
 & &+{\mathbf{B}}_n^\top\mathbf{J}_{n+1}^\top\Lambda_{n+1}\mathbf{J}_{n+1}\bar{\mathbf{B}}_n+\bar{\mathbf{B}}_n^\top\mathbf{J}_{n+1}^\top\Lambda_{n+1}\mathbf{J}_{n+1}{\mathbf{B}}_n \;,\\
\mathbf{N}_n  & = & D_n^\top\mathbf{J}_{n+1}^\top\Theta_{n+1} \mathbf{J}_{n+1}D_n\;,\\
\mathbf{S}_n   & = & (\mathbf{B}_n+\bar{\mathbf{B}}_n)^\top\mathbf{J}_{n+1}^\top\Theta_{n+1} \mathbf{J}_{n+1}D_n+\bar {\mathbf{B}}_n^\top \mathbf{J}_{n+1}^\top\Lambda_{n+1}\mathbf{J}_{n+1} D_n\;.
\enqs
Since $\mu\in \Pc_{\textrm{cl}}\big((E\times F)^n)$, there exists $\bar a\in \tilde{\mathfrak{C}}_{\textrm{cl}}(\P_{(X_0^a,Y_0^a)})$ such that $\mu$ is the law of $(X^a_{0:n},Y^a_{0:n})$. Therefore, we get
\beqs
    \Big(\int_{(\R^{2d})^n}a_n(y_{0:n})d\mu(x_{0:n},y_{0:n})\Big)^\top    \mathbf{N}_n\int_{(\R^{2d})^n}a_n(y_{0:n})d\mu(x_{0:n},y_{0:n}) & & \\
     + 2{}_{n}\bar{\mu}^\top 
    \mathbf{S}_n  \int_{(\R^{2d})^n} a_n(y_{0:n})d\mu(x_{0:n},y_{0:n}) & &  \\
      + 2\int_{(\R^{2d})^n} x_n^{\top} {B}_n^\top\mathbf{J}_{n+1}^\top\Lambda_{n+1}\mathbf{J}_{n+1} D_na_n(y_{0:n})d\mu(x_{0:n},y_{0:n}) & & \\
  + \int_{(\R^{2d})^n} a_n(y_{0:n})^\top \big({D}_n^\top\mathbf{J}_{n+1}^\top\Lambda_{n+1}\mathbf{J}_{n+1} D_n+ R_n\big)a_n(y_{0:n})d\mu(x_{0:n},y_{0:n}) & = & \\
  \E\Big[a_n(Y_{0:n}^{\bar a})\Big] ^\top    \mathbf{N}_n \E\Big[a_n(Y_{0:n}^{\bar a})\Big] & & \\
   + 2{}_{n}\bar{\mu}^\top 
    \mathbf{S}_n \E\Big[a_n(Y_{0:n}^{\bar a})\Big] 
    & &  \\
      + 2  \E\Big[ (X_n^{\bar a})^{\top} {B}_n^\top\mathbf{J}_{n+1}^\top\Lambda_{n+1}\mathbf{J}_{n+1} D_na_n(Y_{0:n}^{\bar a})\Big] & & \\
       + \E\Big[ a_n(Y_{0:n}^{\bar a})^\top \big({D}_n^\top\mathbf{J}_{n+1}^\top\Lambda_{n+1}\mathbf{J}_{n+1} D_n+ R_n\big)a_n(Y_{0:n}^{\bar a})\Big] &  &
\enqs
for any $a_n\in \tilde{\mathfrak{C}}_n(\mu)$. 
Since $ \Lambda_{n+1}$ and $R_n$ are positive and from the definition of $\mathbf{S}_n$, we can apply Jensen conditional inequality given $Y_n$ and we get 
\beqs
 \E\Big[a_n(Y_{0:n}^{\bar a})\Big] ^\top    \mathbf{N}_n \E\Big[a_n(Y_{0:n}^{\bar a})\Big] 
   +  2{}_{n}\bar{\mu}^\top 
    \mathbf{S}_n \E\Big[a_n(Y_{0:n}^{\bar a})\Big] 
    & &  \\
      + 2  \E\Big[ (X_n^{\bar a})^{\top} {B}_n^\top\tilde \Lambda_{n+1} D_na_n(Y_{0:n}^{\bar a})\Big] & & \\
       + \E\Big[ a_n(Y_{0:n}^{\bar a})^\top \big({D}_n^\top\mathbf{J}_{n+1}^\top\Lambda_{n+1}\mathbf{J}_{n+1} D_n+ R_n\big)a_n(Y_{0:n}^{\bar a})\Big] & \geq &\\
        \E\Big[\hat a_n(Y_{n}^{\bar a})\Big] ^\top    \mathbf{N}_n \E\Big[\hat a_n(Y_{n}^{\bar a})\Big] 
   +  2{}_{n}\bar{\mu}^\top 
    \mathbf{S}_n \E\Big[\hat a_n(Y_{n}^{\bar a})\Big] 
    & &  \\
      + 2  \E\Big[ (X_n^{\bar a})^{\top} {B}_n^\top\mathbf{J}_{n+1}^\top\Lambda_{n+1}\mathbf{J}_{n+1} D_n\hat a _n(Y_{n}^{\bar a})\Big] & & \\
       + \E\Big[ \hat a_n(Y_{n}^{\bar a})^\top \big({D}_n^\top\mathbf{J}_{n+1}^\top\Lambda_{n+1}\mathbf{J}_{n+1} D_n+ R_n\big)\hat a_n(Y_{n}^{\bar a})\Big] &  &
\enqs
where $\hat a_n$ is defined by
\beqs
 \hat a_n(y) & = & \E\Big[ a_n(Y_{0:n}^{\bar a})\Big|Y_n=y\Big]\;,\quad y\in\R^d\;.
\enqs
Therefore, the infimum in the definition of $W_n$ can be restricted to  $\tilde{\mathfrak{C}}_{n,\textrm{cl}}(\mu)$: 
\beqs
W_n(\mu) & = &  \inf_{a_n\in\tilde{\mathfrak{C}}_{n,\textrm{cl}}(\mu)} 
  \Big\{\langle{}_{n}\mu\rangle(\mathbf{Q}_n+\mathbf{B}_n^\top\mathbf{J}_{n+1}^\top\Lambda_{n+1}\mathbf{J}_{n+1}\mathbf{B}_n)  +({}_{n}\bar{\mu})^\top \mathbf{K}_n
  ({}_{n}\bar{\mu})\\
   & & \qquad \qquad+ \textrm{Tr}(\mathbf{J}_{n+1}^\top\Lambda_{n+1}\mathbf{J}_{n+1})+\textrm{Tr}(\mathbf{I}^\top \Lambda_{n+1}\mathbf{I})+\chi_{n+1}\\
   & &  \qquad \qquad+ \Big(\int_{\R^{2d}}a_n(y_{n})d{}_{n}\mu(x_{n},y_{n})\Big)^\top 
   \mathbf{N}_n\int_{\R^{2d}}a_n(y_{n})d{}_n\mu(x_{n},y_{n})\\
    & & \qquad \qquad+ 2{}_{n}\bar{\mu}^\top 
    \mathbf{S}_n
    \int_{\R^{2d}} a_n(y_{n})d{}_n\mu(x_{n},y_{n})  \\
     & & \qquad \qquad+ 2\int_{\R^{2d}} x_n^{\top} {B}_n^\top\mathbf{J}_{n+1}^\top\Lambda_{n+1}\mathbf{J}_{n+1} D_na_n(y_{n})d{}_{n}\mu(x_{n},y_{n}) \\
 & & \qquad \qquad+ \int_{\R^{2d}} a_n(y_{n})^\top \big({D}_n^\top\mathbf{J}_{n+1}^\top\Lambda_{n+1}\mathbf{J}_{n+1} D_n+ R_n\big)a_n(y_{n})d{}_n\mu(x_{n},y_{n})\Big\}\;.
 \enqs
From this last identity we deduce that $W_n$ depends only on ${}_n\mu$.
Using \eqref{dynYLQ}, we have
\beqs
W_n(\mu)  & = &  \inf_{a_n\in\tilde{\mathfrak{C}}_{n,\textrm{cl}}(\mu)} 
  \Big\{\langle{}_{n}\mu\rangle(\mathbf{Q}_n+\mathbf{B}_n^\top\Lambda_{n+1}\mathbf{B}_n)  + \textrm{Tr}(\mathbf{I}^\top \Lambda_{n+1}\mathbf{I})+\chi_{n+1}+({}_{n}\bar{\mu})^\top \mathbf{K}_n
  ({}_{n}\bar{\mu})\\
   & &  \qquad \qquad+ \Big(\int_{\R^{d}}a_n(y_{n})d\,{}^2_{n}\mu(y_{n})\Big)^\top 
   \mathbf{N}_n\int_{\R^{d}}a_n(y_{n})d\,{}^2_{n}\mu(y_{n})\\
    & & \qquad \qquad+ 2\,{}_{n}\bar{\mu}^\top 
    \mathbf{S}_n
    \int_{\R^{2d}}a_n(y_{n})d\,{}^2_{n}\mu(y_{n})  \\
     & & \qquad \qquad+ 2\int_{\R^{d}} \Phi_n(\mu,y_n)^{\top} {B}_n^\top\mathbf{J}_{n+1}^\top\Lambda_{n+1}\mathbf{J}_{n+1} D_n a_n(y_{n})d\;{}^2_n\mu(y_{n}) \\
 & & \qquad \qquad+ \int_{\R^{d}} a_n(y_n)^\top \big({D}_n^\top\mathbf{J}_{n+1}^\top\Lambda_{n+1}\mathbf{J}_{n+1} D_n+ R_n\big)a_n(y_n)d\;{}^2_n\mu(y_{n})\Big\}\;.
\enqs
where $\Phi_n$  is  given by 
\eqref{defPhi}.
We then notice that the function of $a_n$ inside the infimum is continuous and goes to $+\infty$ as $\int |a_n |^2d\;{}^2_n\mu$ goes to infinity since $\Lambda_{n+1}+\Theta_{n+1}$ is nonnegative and $R_n$ is positive. Hence, this function admits a global minimum. 
Since this function is continuously differentiable we can compute the first order condition and we get 

\beq\label{FOC}
 \mathbf{N}_n\Big(\int_{\R^{2d}}a_n(y'_{n})d\,{}^2_{n}\mu(y'_{n})\Big) 
   +   \mathbf{S}_n^\top 
    {}_{n}\bar{\mu} +D_n^{\top} \mathbf{J}_{n+1}^\top\Lambda_{n+1}\mathbf{J}_{n+1}{B}_n \Phi(\mu,y_n)& & \\+ \big({D}_n^\top\mathbf{J}_{n+1}^\top\Lambda_{n+1}\mathbf{J}_{n+1} D_n+ R_n\big)a_n(y_{n}) & = & 0\;,\quad y_n\in \R^d\;.\nonumber
\enq
%
Therefore we get
\beqs
a^*_n(y_{n}) & = & -\big({D}_n^\top\mathbf{J}_{n+1}^\top\Lambda_{n+1}\mathbf{J}_{n+1} D_n+ R_n\big)^{-1}\big({D}_n^\top\mathbf{J}_{n+1}^\top\Lambda_{n+1}\mathbf{J}_{n+1}  {B}_n\Phi(\mu,y_n)\\
 & & \quad\qquad\qquad\quad\qquad\quad\qquad\qquad\qquad + \mathbf{S}_n^\top 
    {}_{n}\bar{\mu}+ \mathbf{N}_n \int_{\R^{d}}a_n(y_{n})d({}^2_n\mu)(y_{n})\big)\;.
\enqs
Taking the integral with respect to $({}^2_n\mu)$ on both sides of  \eqref{FOC}, we get
\beqs
\int_{\R^{d}}a_n(y_{n})d({}^2_n\mu)(y_{n}) & =& -\big({D}_n^\top\mathbf{J}_{n+1}^\top(\Lambda_{n+1}+ \Theta_{n+1})\mathbf{J}_{n+1} D_n+ R_n\big)^{-1}  
   \tilde{\mathbf{S}}_n^\top 
    {}_{n}\bar{\mu}
\enqs
with
\beqs
 \tilde{\mathbf{S}}_n & = &  {\mathbf{S}}_n+\mathbf{B}_n^\top\mathbf{J}_{n+1}^\top\Lambda_{n+1}\mathbf{J}_{n+1}{D}_n
\enqs
and
\beqs
a^*_n(y_{n}) & = & G_n\Big(\Xi_n \Phi(\mu,y_n)+ \hat{\mathbf{S}}_n^\top 
    {}_{n}\bar{\mu}
 \Big)\;. 
\enqs
with
\beqs
G_n & = & -\big({D}_n^\top\mathbf{J}_{n+1}^\top\Lambda_{n+1}\mathbf{J}_{n+1} D_n+ R_n\big)^{-1}\\
\Xi_n & = & {D}_n^\top\mathbf{J}_{n+1}^\top\Lambda_{n+1}\mathbf{J}_{n+1}  {B}_nJ_{n+1}^{-1}\\
\hat{\mathbf{S}}_n & = & {\mathbf{S}_n}-\tilde{\mathbf{S}}_n \big({D}_n^\top\mathbf{J}_{n+1}^\top(\Lambda_{n+1}+ \Theta_{n+1})\mathbf{J}_{n+1} D_n+ R_n\big)^{-1}\mathbf{N}_n\;.
\enqs

We then get
\beqs
W_n(\mu) & = &   \langle{}_{n}\mu\rangle(\mathbf{Q}_n+\mathbf{B}_n^\top\Lambda_{n+1}\mathbf{B}_n) + \textrm{Tr}(\mathbf{J}_{n+1}^\top\Lambda_{n+1}\mathbf{J}_{n+1})+\textrm{Tr}(\mathbf{I}^\top \Lambda_{n+1}\mathbf{I})\\
 & & +\chi_{n+1}+({}_{n}\bar{\mu})^\top \mathbf{K}_n
  ({}_{n}\bar{\mu})\\
   & &   + \Big(\int_{(\R^{2d})}a^*_n(y_{n})d{}_{n}\mu(x_{n},y_{n})\Big)^\top 
   \mathbf{N}_n\int_{(\R^{2d})}a^*_n(y_{n})d\;{}_n\mu(x_n,y_{n})\\
    & &  + 2{}_{n}\bar{\mu}^\top 
    \mathbf{S}_n
    \int_{\R^{2d}} a^*_n(y_{n})d\;{}_n\mu(x_n,y_{n})  \\
     & &  + 2\int_{(\R^{2d})} x_n^{\top} {B}_n^\top\mathbf{J}_{n+1}^\top\Lambda_{n+1}\mathbf{J}_{n+1} D_na^*_n(y_{n})d\;{}_n\mu(x_n,y_{n}) \\
 & &  + \int_{\R^{2d}} a^*_n(y_{n})^\top \big({D}_n^\top\mathbf{J}_{n+1}^\top\Lambda_{n+1}\mathbf{J}_{n+1} D_n+ R_n\big)a^*_n(y_{n})d\;{}_n\mu(x_n,y_{n})\\
 & = & \langle{}_{n}\mu\rangle( \Lambda_n ) +  {}_n\bar \mu^\top \Theta_n\; {}_n\bar \mu +  \chi_n
\enqs
with
\beqs
\Lambda_n & = &\mathbf{Q}_n+\mathbf{B}_n^\top\Lambda_{n+1}\mathbf{B}_n +\mathbf{I}\Xi_n^\top({D}_n^\top\mathbf{J}_{n+1}^\top\Lambda_{n+1}\mathbf{J}_{n+1} D_n+ R_n)\Xi_n\mathbf{I}^\top\;,\\
\Theta_n & = & \mathbf{K}_n\\
 & & +\hat{\mathbf{S}}_n\big({D}_n^\top\mathbf{J}_{n+1}^\top\Lambda_{n+1}\mathbf{J}_{n+1} D_n+ R_n\big)^{-1}\mathbf{N}_n\big({D}_n^\top\mathbf{J}_{n+1}^\top\Lambda_{n+1}\mathbf{J}_{n+1} D_n+ R_n\big)^{-1}  \hat{\mathbf{S}}_n^\top
   \\
 & & +\mathbf{I}\Xi_n^\top\big({D}_n^\top\mathbf{J}_{n+1}^\top\Lambda_{n+1}\mathbf{J}_{n+1} D_n+ R_n\big)^{-1}\mathbf{N}_n\big({D}_n^\top\mathbf{J}_{n+1}^\top\Lambda_{n+1}\mathbf{J}_{n+1} D_n+ R_n\big)^{-1}  \Xi_n\mathbf{I}^\top
  \\
 & & +\mathbf{I}\Xi_n^\top\big({D}_n^\top\mathbf{J}_{n+1}^\top\Lambda_{n+1}\mathbf{J}_{n+1} D_n+ R_n\big)^{-1}\mathbf{N}_n\big({D}_n^\top\mathbf{J}_{n+1}^\top\Lambda_{n+1}\mathbf{J}_{n+1} D_n+ R_n\big)^{-1}  \hat{\mathbf{S}}_n^\top 
   \\
   & & +\hat{\mathbf{S}}_n\big({D}_n^\top\mathbf{J}_{n+1}^\top\Lambda_{n+1}\mathbf{J}_{n+1} D_n+ R_n\big)^{-1}\mathbf{N}_n\big({D}_n^\top\mathbf{J}_{n+1}^\top\Lambda_{n+1}\mathbf{J}_{n+1} D_n+ R_n\big)^{-1}   \Xi_n\mathbf{I}^\top\\
 & &   +2{\mathbf{S}}_n\big({D}_n^\top\mathbf{J}_{n+1}^\top\Lambda_{n+1}\mathbf{J}_{n+1} D_n+ R_n\big)^{-1} \big( \Xi_n\mathbf{I}^\top+\hat{\mathbf{S}}_n^\top\big) \\
 & & +2\mathbf{I}\big({B}_nJ_{n+1}^{-1}\big)^\top\mathbf{J}_{n+1}^\top\Lambda_{n+1}\mathbf{J}_{n+1} D_n\big({D}_n^\top\mathbf{J}_{n+1}^\top\Lambda_{n+1}\mathbf{J}_{n+1} D_n+ R_n\big)^{-1}\hat{\mathbf{S}}_n^\top\\
& & +\hat{\mathbf{S}}_n\big({D}_n^\top\mathbf{J}_{n+1}^\top\Lambda_{n+1}\mathbf{J}_{n+1} D_n+ R_n\big)^{-1}\hat{\mathbf{S}}_n^\top\\
 & & +\mathbf{I}\Xi_n^\top\big({D}_n^\top\mathbf{J}_{n+1}^\top\Lambda_{n+1}\mathbf{J}_{n+1} D_n+ R_n\big)^{-1}\hat{\mathbf{S}}_n^\top
   \\ & &
  +\hat{\mathbf{S}}_n\big({D}_n^\top\mathbf{J}_{n+1}^\top\Lambda_{n+1}\mathbf{J}_{n+1} D_n+ R_n\big)^{-1}\Xi_n\mathbf{I}^\top\;, \\
\chi_n & = & \textrm{Tr}(\mathbf{J}_{n+1}^\top\Lambda_{n+1}\mathbf{J}_{n+1})+\textrm{Tr}(\mathbf{I}^\top \Lambda_{n+1}\mathbf{I})+\chi_{n+1}\;.
\enqs
We then easily have $\Lambda_n$ nonnegative.
A straghtforward computation shows that $\Theta_n+\Lambda_n$ is also nonnegative. 
Therefore, the induction property holds true at rank $n$, and it holds for any $n=0,\ldots,T$. \end{proof}

\subsection{Numerical approximation of the optimal value}
\subsubsection{The algorithm} \label{quantize_algo}
We present a numerical algorithm for the approximation of the optimal value base on the verification Theorem \ref{VerifTHM}.  For that, we define two finite subsets $\Lambda_E$ and $\Lambda_F$ of $E$ and $F$ respectively by
\beqs
\Lambda_E & = & \big\{x^1,\ldots,x^N \big\}\;,\\
\Lambda_F & = & \big\{y^1,\ldots,y^N \big\}\;.
\enqs
We next introduce two Voron\"i tessellations   $(C(\Lambda_E)^i)_{1\leq i\leq N}$ and  $(C(\Lambda_F)^i)_{1\leq i\leq N}$ of subsets of $E$ and $F$ respectively. This means that  $(C(\Lambda_E)^i)_{1\leq i\leq N}$ and  $(C(\Lambda_F)^i)_{1\leq i\leq N}$ satisfy
\beqs
\bigcup_{1\leq i\leq N}C(\Lambda_E)^i ~ = ~ E \;,& & C(\Lambda_E)^i\cap C(\Lambda_E)^j ~=~\emptyset \mbox{ for } i\neq j\;,\\
\bigcup_{1\leq i\leq N}C(\Lambda_F)^i ~= ~ F\;,& & C(\Lambda_F)^i\cap C(\Lambda_F)^j ~=~\emptyset \mbox{ for } i\neq j\;,
\enqs
and
\beqs
C(\Lambda_E)^i & \subset & \big\{x\in E~:~|x-x^i|=\min_{1\leq j\leq N}|x-x^j|\big\}\;,\\
C(\Lambda_F)^i & \subset & \big\{y\in F~:~|y-y^i|=\min_{1\leq j\leq N}|x-x^j|\big\}\;,\\
\enqs
for ${i=1,\ldots, N}$. We then define the projection operators $\textrm{Proj}_{\Lambda_E}$ and $\textrm{Proj}_{\Lambda_F}$ by 
\beqs
\textrm{Proj}_{\Lambda_E}(x) & = & x^i
\enqs
for $x\in C^i(\Lambda_E)$ and 
\beqs
\textrm{Proj}_{\Lambda_F}(y) & = & y^i
\enqs
for $y\in C^i(\Lambda_F)$ and $1\leq i\leq N$. Our goal is to provide a discrete version of the dynamic programming equation defining the functions $W_0,\ldots,W_T$ in Theorem \ref{VerifTHM}. 
We first discretize the initial conditions $\xi$ and $\zeta$ by defining
\beqs
\hat \xi & = & \textrm{Proj}_{\Lambda_E}(\xi)\;,\\
\hat \zeta & = & \textrm{Proj}_{\Lambda_F}(\zeta)\;.
\enqs
We then define the processes $(\hat X^a,\hat Y^a)$ as the quantizer $(X^a,Y^a)$ according to $\Lambda_E$ and $\Lambda_F$. This means that $(\hat X^a,\hat Y^a)$ is the approximation of $(X^a,Y^a)$ valued in $\Lambda_E\times\Lambda_F$ by $(\hat X^a_0,\hat Y^a_0)=(\hat \xi,\hat \zeta)$ and 
\beqs
\hat X ^a_{n+1} & = & \textrm{Proj}_{\Lambda_E}\Big(G_{n+1}\big( \hat X^a_n,\P_{\hat X^a_n},a(\hat Y^a_{0:n}),\eps_{n+1} \big)\Big)\;, \\
\hat Y ^a_{n+1} & = & \textrm{Proj}_{\Lambda_F}\Big(H_{n+1}\big( \hat X^a_n,\hat Y^a_n,a(\hat Y^a_{0:n}),\eta_{n+1} \big)\Big) 
\enqs
for $n=0,\ldots,T-1$. A straightforward computation give the dynamics of the process $(\hat X^a,\hat Y^a)$ as
\beqs
\P\big( (\hat X^a_{n+1},\hat Y^a_{n+1})=( x^i,  y^j)\big| \hat X^a_{0:n},\hat Y^a_{0:n} \big) & = & \hat h_{n+1}\big( x^i,\hat Y^a_{n},a_n(\hat Y^a_{0:n}), y^j \big)\hat P_{n+1}^{a_n(\hat Y^a_{0:n})}\big(\hat X ^a_n,\P_{\hat X ^a_n},\hat x^j\big)
\enqs
 for $n=0,\ldots,T-1$ and $( x^i, y^j)\in\Lambda_E\times\Lambda_F$, where
\beqs
\hat h_{n+1}\big( x^i,\hat Y^a_{n},a_n(\hat Y^a_{0:n}), y^j \big) & = & \int_{C_j(\Lambda_F)}h_{n+1}\big( x^i,\hat Y^a_{n},a_n(\hat Y^a_{0:n}),y \big)dy
\enqs
and
\beqs
\hat P_{n+1}^{a_n(\hat Y^a_{0:n})}\big(\hat X ^a_n,\P_{\hat X ^a_n}, x^j\big) 
 &  = & \int_{C_i(\Lambda_E)} P_{n+1}^{a_n(\hat Y^a_{0:n})}\big(\hat X ^a_n,\P_{\hat X ^a_n},dx\big)\;.
\enqs
Then, the global controlled transition $\bar P$ is replaced by $\hat{\bar{P}}$  defined by
\beqs
\hat{\bar{P}}_{n+1}^{a_n}( z, M, w) & = & \hat P _{n+1}^{a_n(y_{0:n})}( x_n,{}^1_{n} M,w_{1})\hat h_{n+1}(w_1,y_{n}, a_n(y_{0:n}), w_2)
\enqs
for $n=0,\ldots,T-1$, $z=(x_{\ell},y_{\ell})_{0\leq\ell\leq n}\in (\Lambda_E\times \Lambda_F)^{n+1}$, $M\in \Pc_2\big((\Lambda_E\times \Lambda_F\big)^{n+1}\big)$, $a_{n}\in \tilde{\mathfrak{C}}_{n}(M)$ and $w=(w_1,w_2)\in \Lambda_E\times \Lambda_F$. For $n=0,\ldots,T-1$ and $\mu\in \Pc_2((\Lambda_E\times \Lambda_F)^{n+1})$  $a_n\in \tilde{\mathfrak{C}}_n(\mu)$, we define  the measure $\bar P^{a}_{n+1}\mu \in \Pc_2((\Lambda_E\times \Lambda_F)^{n+2})$ by
\beqs
\hat{ \bar P}^{a_n}_{n+1}\mu (x_{0:n+1},y_{0:n+1}) & = & \hat{\bar P}^{a_n}_{n+1}(x_{0:n},y_{0:n}, \mu,x_{n+1},y_{n+1})\mu(x_{0:n},y_{0:n})
\enqs
for $(x_{0:n+1},y_{0:n+1})\in (\Lambda_E\times \Lambda_F)^{n+2}$. The dynamics of the controlled measures $\P_{\hat X_n^a,\hat Y_n^a}$ can be written under the following simplified form
\beqs
\P_{\hat X_{0:n+1}^a,\hat Y_{0:n+1}^a}(x_{0:n+1},y_{0:n+1}) & = &\hat{\bar P}^{a_n}_{n+1}\P_{\hat X_{0:n}^a,\hat Y_{0:n}^a}(x_{0:n},y_{0:n})\;.
\enqs
We then define the related discretized cost coefficients $\hat C_n$ for 
$n=0,\ldots,T-1$ and $\gamma$ by
\beqs
\hat C_n(M,a_n) & = & \sum_{(x_{0:n},y_{0:n})\in\Lambda_E\times\Lambda_F} c_n\Big(x_n,{}^1_nM,a_n(y_{0:n})\Big)M(x_{0:n},y_{0:n})
\enqs
for $M\in\Pc_2\big((\Lambda_E\times \Lambda_F)^{n+1}\big)$, $a_n\in\tilde{\mathfrak{C}}_n(M)$
and 
\beqs
\hat \Gamma (M) &  = & \sum_{(x_{0:T},y_{0:T})\in\Lambda_E\times\Lambda_F} \gamma\Big(x_T,{}^1_TM\Big)dM_T(x_{0:T},y_{0:T})\;,
\enqs
for  $M\in\Pc_2\big((\Lambda_E\times \Lambda_F)^{T+1}\big)$. Then, the related approximated value functions $\hat W_n$, $n=0,\ldots,T$ are given by
\beqs
\hat W_T(\mu) & = & \hat \Gamma (\mu)\;,\quad \mu\in \Pc_{2}((\Lambda_E\times \Lambda_F)^{T+1})\;,
\enqs
and 
\beqs
\hat W_n(\mu) & = & \inf_{a_n\in\tilde{\mathfrak{C}}_{n}(\mu)} \big\{\hat C_n\big(\mu,a_n\big)+\hat W_{n+1}( \hat{\bar P}^{a_n}_{n+1}\mu)\big\}\;,\quad \mu\in \Pc_{2}((\Lambda_E\times \Lambda_F)^{n+1})\;,
\enqs
 for $ n=0,\ldots, T-1$.
To get tractable versions  $\tilde W_n$ of values functions $\hat W_n$, $n=0,\ldots,T$, we define finite subsets $\Pi_{n+1}$ of $\Pc_2((\Lambda_E\times\Lambda_F)^{n+1})$ as follows. We first introduce $L$ sequences $(p_{i,j}^\ell)_{1\leq i,j\leq N}$ for $\ell=1,\ldots,L$ such that $p_{i,j}^\ell\in\R_+$ for $i,j=1,\ldots,N$ and
\beqs
\sum_{1\leq i,j\leq N}p_{i,j}^\ell & =& 1
\enqs
for $\ell=1,\ldots,L$. 
We next define the set $\Pi_n$ by
\beqs
\Pi_n & =  & \Big\{ \sum_{1\leq i_1,j_1,\ldots,i_n,j_n\leq N}\delta_{(x^{i_0},\ldots,x^{i_n},y^{i_0},\ldots,y^{i_n})} \prod_{r=0}^np^{\ell_r}_{i_r,j_r}, ~\ell_1, \ldots, \ell_n \in \{1, \ldots, L\}\Big\}\\
 & \subset & \Pc_2\big( (\Lambda_E\times\Lambda_F)^n \big)
\enqs
 for $n=1,\ldots,T$. We next define the approximation  $\tilde{\bar P}$ of $\hat{\bar{P}}$  defined by
\beqs
\tilde{\bar{P}}_{n+1}^{a_n}( z, M, w) & = & \textrm{Proj}_{\{ p^1_{i,j}\ldots,p^L_{i,j} \}}\Big(\hat{\bar{P}}_{n+1}^{a_n}( z, M, w)\Big)
\enqs
for $n=0,\ldots,T-1$, $z=(x_{\ell},y_{\ell})_{0\leq\ell\leq n}\in (\Lambda_E\times \Lambda_F)^{n+1}$ with $(x_n,y_n)=(x^i,y^j)$,   $M\in \Pc_2\big((\Lambda_E\times \Lambda_F\big)^{n+1}\big)$, $a_{n}\in \tilde{\mathfrak{C}}_{n}(M)$ and $w=(w_1,w_2)\in \Lambda_E\times \Lambda_F$. We observe that 
\beqs
\tilde{\bar{P}}_{n+1}^{a_n}\mu & \in & \Pi_{n+2}
\enqs
for $n=0,\ldots,T-1$, $\mu\in \Pi^{n+1}$ and $a_{n}\in \tilde{\mathfrak{C}}_{n}(\mu)$. In particular, the functions $\tilde W_n$, $n=0,\ldots,T$ are given by
\beqs
\tilde W_T(\mu) & = & \hat \Gamma (\mu)\;,\quad \mu\in \Pi_{T+1}\;,
\enqs
and 
\beqs
\tilde W_n(\mu) & = & \inf_{a_n\in\tilde{\mathfrak{C}}_{n}(\mu)} \big\{\hat C_n\big(\mu,a_n\big)+\tilde W_{n+1}( \tilde{\bar P}^{a_n}_{n+1}\mu)\big\}\;,\quad \mu\in \Pi_{n+1}\;,
\enqs
 for $ n=0,\ldots, T-1$, provide a computable approximation of the functions $W_n$, $n=0,\ldots,T$.

\subsubsection{Test of the algorithm}

It is possible to compare this approximated algorithm with the linear quadratic case developed earlier.
We choose to fix $\Lambda_E$ and $\Lambda_F$ to be the centers of the Vorono\"i cells $C(\Lambda_E)^i$ and $C(\Lambda_F)^i$ for ${1\leq i\leq N}$ respectively.  
Precisely, we apply Lloyd Algorithm on $\mathcal{N}(0, I_d)$ to obtain $C(\Lambda_E)^i$ and  $C(\Lambda_F)^i$.   Thus, $\Lambda_E = \Lambda_F$ and $C(\Lambda_E)^i = C(\Lambda_F)^i$. We refer to the book  \cite{pages2018numerical} for a description of quantization methods and their related algorithms.

We fix $d=2$. In order to be complete, we describe all the matrices we used for this experiment.

For the main dynamics we fix:
\beqs 
B_k  &=& \begin{pmatrix}
0 & 0 \\
0 & 0
\end{pmatrix},\\ 
\bar{B}_k  &=& \begin{pmatrix}
0 & 0 \\
0 & 0
\end{pmatrix}, \\
D_k &=& \begin{pmatrix}
1 & 1 \\
0 & 1
\end{pmatrix}, \\
J_{k+1} &=& \begin{pmatrix}
1 & 1 \\
0 & 1
\end{pmatrix}
\enqs for ${k=1,\ldots, T}$.
For the cost function $\bar{J}$, we fix:
\beqs 
Q_k &=& \begin{pmatrix}
1 & 1 \\
1 & 1
\end{pmatrix}, \\
\bar{Q}_k &=& \begin{pmatrix}
1 & 1 \\
1 & 1
\end{pmatrix}, \\
R_k &=& I_d
\enqs for  ${k=1,\ldots , T}$.
We approximate the integral needed to compute $\hat h_{n+1}$ and $\hat P_{n+1}^{a_n(\hat Y^a_{0:n})}$ through Monte Carlo simulations. Controls $\alpha$ are restricted to the following set
\begin{equation*}
\tilde{\mathfrak{C}}_{n}(M):= \left\{\begin{pmatrix}
-2 \\
-2
\end{pmatrix}, \begin{pmatrix}
-1 \\
-1
\end{pmatrix}, \begin{pmatrix}
1 \\
1
\end{pmatrix}, \begin{pmatrix}
2 \\
2
\end{pmatrix}\right\}.
\end{equation*}
Finally, we fix $T=3$ and we compute $W_t$ at time $t = 0$ using the formula derived in the proof of Proposition \ref{kalman_bucy_proposition} and $\tilde{W}_0$ using the algorithm described in Section \ref{quantize_algo}. The relative error $\frac{|\tilde{W}_0 - W_0|}{W_0}$ is computed for $N=2,4,10$ and $20$. The results are presented in the Figure \ref{fig:quantiz_results_mckean}.

\begin{center}
	\begin{figure}[H]
		\centering
		\includegraphics[scale = 0.62]{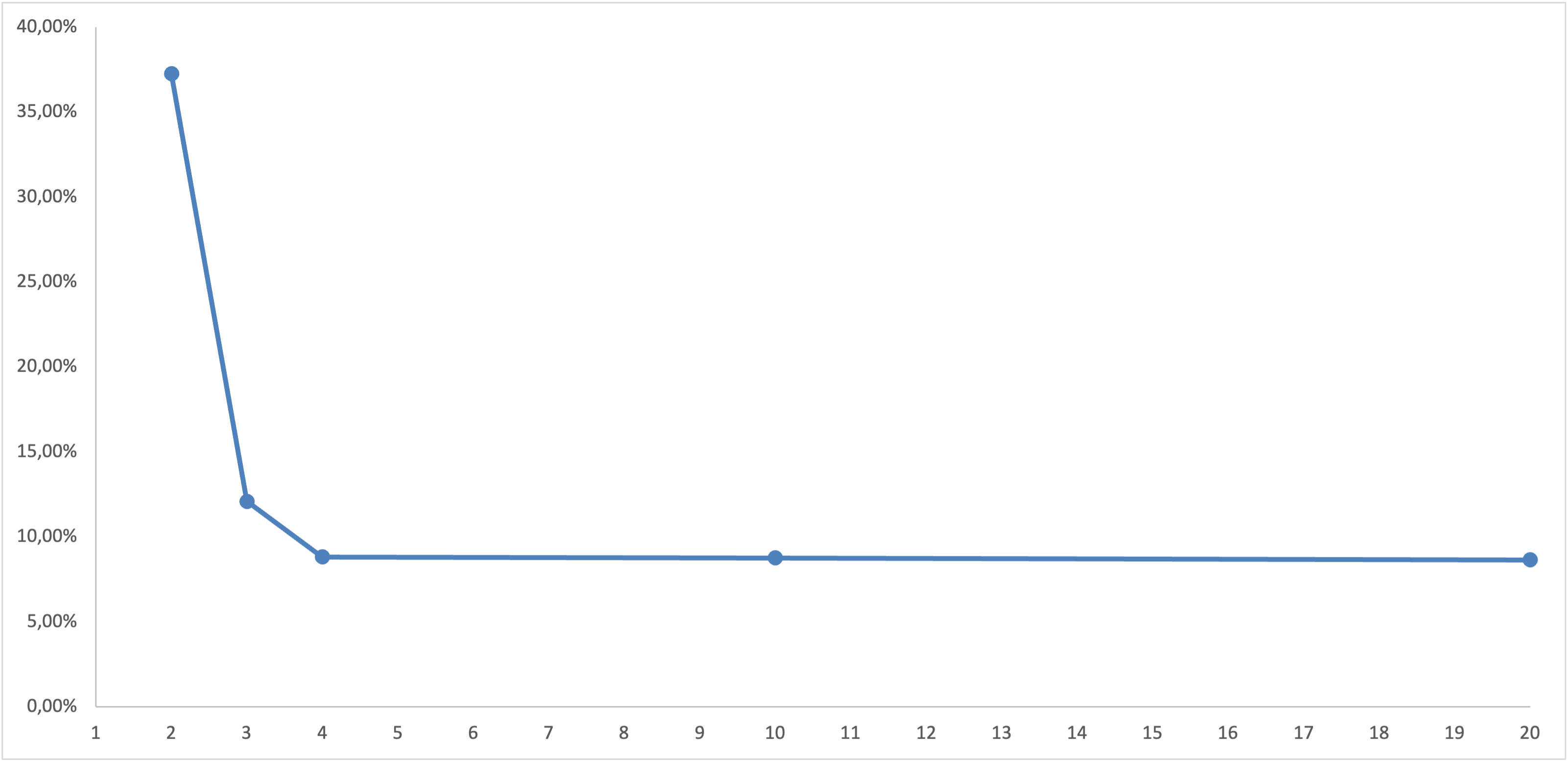}
		\caption{Relative Error of Value function as $N$ grows}
		\label{fig:quantiz_results_mckean}
	\end{figure}
\end{center}

Errors are expressed as percentages. Increasing $N$ decreases the relative error. For small $N$, we achieve good results. For example, with $N\geq4$, the absolute error is less than 10\% (round 8\%). However, the error does not decrease anymore for $N\geq 4$. This might be explained by the restriction done on the control set.
 
\subsection{A Mean-Variance optimal investment problem}
We refer to \cite{pham2011portfolio} for a review on portfolio optimization in partial observation framework. 
 We consider a financial market over the horizon $\{0,\ldots,T\}$. We suppose that this market is composed by one asset with return process $(R_0,\ldots,R_T)$  satisfying
\beqs 
R_k =  b_0\Delta + \sigma\sqrt{\Delta}\epsilon_{k+1}\;,\quad k=0,\ldots,T-1\;,
\enqs 
where the drift and the variance $b_0$ and $\sigma$ are known. We assume that $(\epsilon_1,\ldots,\epsilon_T)$ is a sequence of independent $\Nc(0,1)$-distributed random variables independent from $X_0$. We consider  an investor who can invest at each time on this asset. The wealth process  $(X_0^\alpha, \ldots, X_T^\alpha)$ is given by
\beqs 
X_{k+1}^\alpha = X_k^\alpha  + \alpha_k \bigg( b_0\Delta + \sigma\sqrt{\Delta}\epsilon_{k+1}\bigg) \;,\quad k=0,\ldots,T-1\;,
\enqs 
Due to a lack of liquidity on this asset, we suppose that the investor does not observe its portfolio directly, but rather an approximate representation $(X_0^\alpha, \ldots, X_T^\alpha)$ is given by
\beqs
Y_{k}^\alpha = X_{k}^\alpha + \eta_{k}\;,\quad k=1,\ldots,T\;.
\enqs
Such a situation can be faced by private equity investors since they only have some intuition about the value of their portfolios. We refer to \cite{gompers1997risk}, \cite{gourier2022capital}, \cite{tegondap2018} for more details. The studied system is then:
\beq \label{markowitz_mckean}
\left\{ 
    \begin{array}{ll}
       X_{k+1}^\alpha = X_k^\alpha  + \alpha_k \bigg( b_0\Delta + \sigma\sqrt{\Delta}\epsilon_{k+1}\bigg)\;,  \\
        Y_{k+1}^\alpha = X_{k+1}^\alpha + \eta_{k+1} \;, ~~~~~~ k=0,\ldots,T-1\;.
    \end{array}
\right.
\enq 
We suppose that $\eta_{k+1}$ is a sequence of independent $\Nc(0,1)$-distributed random variables, also independent from $(\eta_1,\ldots,\eta_T)$, $X_0$ and $Y_0$.

The investor goal is to find a portfolio allocation  that minimize a mean-variance criterum:
\beqs
V_0 &=& \inf_{\alpha\in \bar{\mathfrak{C}}}J(\alpha) \\
    &=& \inf_{\alpha\in \bar{\mathfrak{C}}} \Bigg[\frac{\gamma}{2} Var[X_T^\alpha] - \E[X_T^\alpha]\Bigg]
\enqs 
for a $\gamma > 0$.

We propose to apply the numerical approximation algorithm of Section \ref{quantize_algo}. As a starting point, we set the $d=1$, $T=5$. Additionally, we fix $b_0 = 0.02, \sigma = 0.05$ for the Return process, so that the investor can expect a return of $0.02$ with a volatility of $0.05$ for this asset.  Furthermore, we assume the control space to be $\{ 0.5, 0.75, 1, 2\}$. Investing only a quarter, half, three quarters, or the entire wealth of an investor is permitted. Using the proposed algorithm with $N=2$, we compute the optimal control at every time $t=0, \ldots, T$. This control leads to a portfolio. Our proposed allocation is benchmarked against two strategies. The first one is known as 'buy and hold': the investor remains invested in the asset at all times.  For the second strategy, it can be viewed as a classical trending strategy: if the asset return is positive at a given time $t$, the investor invests, else the investor shorts the asset. Based on 250 trajectory simulations, we compute the empirical final wealth mean, denoted by $\bar{\E}[X_T^\alpha]$ and the empirical final wealth variance denoted by $\bar{\text{Var}}[X_T^\alpha]$. We present in Tables \ref{tab:my-table_gamma2}, \ref{tab:my-table_gamma4}, \ref{tab:my-table_gamma8} and \ref{tab:my-table_gamma16} the results for several values of $\gamma$.

\begin{table}[H]

\centering

\resizebox{\textwidth}{!}{%
\begin{tabular}{|l|l|l|l|}
\hline
\textbf{} & \textbf{Proposed Strategy} & \textbf{Buy and Hold} & \textbf{Trending Strategy} \\ \hline
$\bar{\E}[X_T^\alpha]$    & 1,02027868 & \textbf{1,04535767} & 1,01155139 \\ \hline
$\bar{\text{Var}}[X_T^\alpha]$ & \textbf{0,00481573} & 0,00680738 & 0,00688046 \\ \hline
\textbf{$V_0$}       & -1,01546295 & \textbf{-1,03855029}& -1,00467093 \\ \hline
\end{tabular}%
}
\caption{Empirical Results for $\gamma = 2$}
\label{tab:my-table_gamma2}
\end{table}
\begin{table}[H]
\centering

\resizebox{\textwidth}{!}{%
\begin{tabular}{|l|l|l|l|}
\hline
\textbf{} & \textbf{Proposed Strategy} & \textbf{Buy and Hold} & \textbf{Trending Strategy} \\ \hline
$\bar{\E}[X_T^\alpha]$    & 1,02681514 & \textbf{1,03881421} & 1,01034027 \\ \hline
$\bar{\text{Var}}[X_T^\alpha]$ & \textbf{0,00467356} & 0,00653589 & 0,00636125 \\ \hline
\textbf{$V_0$}       & -1,01746802 & \textbf{-1,02574243} & -0,99761777 \\ \hline
\end{tabular}%
}
\caption{Empirical Results for $\gamma = 4$}
\label{tab:my-table_gamma4}
\end{table}

\begin{table}[H]
\centering

\resizebox{\textwidth}{!}{%
\begin{tabular}{|l|l|l|l|}
\hline
\textbf{} & \textbf{Proposed Strategy} & \textbf{Buy and Hold} & \textbf{Trending Strategy} \\ \hline
$\bar{\E}[X_T^\alpha]$    & 1,02314645 & \textbf{1,03975832} & 1,00942357 \\ \hline
$\bar{\text{Var}}[X_T^\alpha]$ & \textbf{0,00452504} & 0,00651461 & 0,00684134\\ \hline
\textbf{$V_0$}       & -1,00504629 & \textbf{-1,01369988} & -0,98205821 \\ \hline
\end{tabular}%
}
\caption{Empirical Results for $\gamma = 8$}
\label{tab:my-table_gamma8}
\end{table}

\begin{table}[H]
\centering

\resizebox{\textwidth}{!}{%
\begin{tabular}{|l|l|l|l|}
\hline
\textbf{} & \textbf{Proposed Strategy} & \textbf{Buy and Hold} & \textbf{Trending Strategy} \\ \hline
$\bar{\E}[X_T^\alpha]$    & 1,01748989 & \textbf{1,03562335} & 1,01643678 \\ \hline
$\bar{\text{Var}}[X_T^\alpha]$ & \textbf{0,00433524} & 0,00693694 & 0,006672365 \\ \hline
\textbf{$V_0$}       & \textbf{-0,98280797} & -0,98012783 & -0,96305786 \\ \hline
\end{tabular}%
}
\caption{Empirical Results for $\gamma = 16$}
\label{tab:my-table_gamma16}
\end{table}

With the proposed approach, the variance of the final wealth is systematically smaller for every $\gamma$. The buy-and-hold strategy provides the best returns for investors: however, the proposed strategy reduces variance compared to a buy-and-hold strategy. For example, when $\gamma = 2$, the proposed strategy reduces volatility by 41\% compared to buy and hold for a return’s cost of 2.45\%.  We note that the proposed strategy is better in terms of $V_0$ value when $\gamma$ is equal to 16. As a result, we have proposed an interesting allocation that can be used to reduce portfolio risk using only two quantization points.

\section{Conclusion}
A partially observed optimal control problem for a system with mean field discrete time dynamics is presented and solved. We extend the linear-quadratic  case (also known as Kalmann Bucy) to deal with the mean field dependence. We also propose a general algorithmic approach based on optimal quantization to  approximate the optimal value. We check the robustness of the algorithm  empirically with a financial example. Some extensions of the  work can be proposed. A first natural question is the estimation of the error of the proposed algorithm. The extension of the results to the continuous time case can also be addressed. This leads to the question of the approximation of the continuous time case by a discrete-time model using an Euler discretization of the continuous problem.

\bibliographystyle{plain}

\bibliography{BayesDTMKVcontrol}
\end{document}